\documentclass[11pt]{amsart}

\usepackage{graphicx}
\usepackage{epsfig}
\usepackage{amsmath}
\usepackage{autobreak}
\usepackage{amssymb}
\usepackage{bigints}
\usepackage[]{breqn}
\usepackage{tikz}
\usepackage{graphicx}
\usepackage{epsfig}
\usepackage[noadjust]{cite}
\usepackage{verbatim}
\usepackage{xcolor}
\usepackage{tikz}
\usetikzlibrary{matrix}
\usepackage[colorlinks,citecolor=red,urlcolor=blue,bookmarks=false,hypertexnames=true]{hyperref}

\theoremstyle{plain}
\newtheorem{theorem}      {Theorem}      [section]
\newtheorem*{theorem*}    {Theorem \eqref{thm:appl}}
\newtheorem{proposition}  [theorem]  {Proposition}
\newtheorem{corollary}    [theorem]  {Corollary}

\theoremstyle{definition}
\newtheorem{example}      [theorem]  {Example}
\newtheorem{remark}       [theorem]  {Remark}

\def \i{\mbox{${\textnormal{i}}$}}
\def \d{\mbox{${\textnormal{d}}$}}
\def \dz{\mbox{${\stackrel{o}{\textnormal{d}}}$}}
\def \Dz{\mbox{${\stackrel{o}{\Delta}}$}}
\def \nz{\mbox{${\stackrel{o}{\nabla}}$}}
\def \tz{\mbox{${\stackrel{o}{\tau}}$}}

\def \r{\mbox{${\mathbb R}$}}
\def \s{\mbox{${\mathbb S}$}}

\def \f{\mbox{$\varphi$}}

\numberwithin{equation}{section}

\begin{document}

\title[New constructions of biharmonic polynomial maps between spheres]
{New constructions of biharmonic polynomial maps between spheres}

\author{Rare\c s Ambrosie}

\address{Faculty of Mathematics\\ Al. I. Cuza University of Iasi\\
Blvd. Carol I, 11 \\ 700506 Iasi, Romania} \email{rares_ambrosie@yahoo.com}

\thanks{Thanks are due to Stefano Montaldo, Cezar Oniciuc and Ye-Lin Ou for useful comments and remarks.}



\begin{abstract}
In this paper, we study diagonal maps between spheres given by two homogeneous polynomial maps between spheres, defined on the same domain sphere. First we find their bitension field, then we give a method for generating proper biharmonic maps between spheres, relying on harmonic homogeneous polynomial maps of different degrees. Further, we establish a result for constructing proper biharmonic product maps using harmonic homogeneous polynomial maps between spheres.

\end{abstract}

\maketitle

\section{Introduction}

Biharmonic maps represent a natural fourth order generalization of the well-known harmonic maps. As suggested by J. Eells and J.H. Sampson in  \cite{ES64,ES65}, or J. Eells and L. Lemaire in \cite{EL83}, biharmonic maps $\phi:M^m\to N^n$ between two Riemannian manifolds are critical points of the bienergy functional
$$
E_2:C^\infty(M,N)\to \mathbb{R}, \qquad E_2(\phi) = \frac{1}{2}\int_{M}\left|\tau(\phi)\right|^2 \ v_g,
$$
where $M$ is compact and $\tau(\phi) = \textnormal{trace}\nabla d\phi$ is the tension field associated to $\phi$. The tension field $\tau(\phi)$ is a section in the pull-back bundle $\phi^{-1}TN$ and its vanishing characterises the harmonicity of the map $\phi$. In 1986, G.Y. Jiang proved in  \cite{J86,J86-2} that the biharmonic maps are characterized by the vanishing of their bitension field, where the bitension field is given by
$$
\tau_2(\phi) = -\Delta\tau(\phi)-\textnormal{trace} R^N\left(d\phi(\cdot),\tau(\phi)\right)\d\phi(\cdot).
$$
The equation $\tau_2(\phi) = 0$ is called the biharmonic equation and it is a fourth order semilinear elliptic equation.

From the biharmonic equation, any harmonic map is biharmonic, so we are interested in the study of biharmonic maps which are not harmonic, called proper biharmonic. Using a simple Bochner-Weitzenb$\ddot{o}$ck formula, G.Y. Jiang proved in \cite{J86,J86-2} that if $M$ is compact and $N$ has non-positive sectional curvature, then a biharmonic map $\phi$ from $M$ to $N$ has to be harmonic. Over the years, biharmonic maps have been studied in various geometric contexts, with a particular focus on their behavior in spaces of positive curvature, especially in the Euclidean sphere, and they continue to be a subject of interest (see, for example, \cite{B24}, \cite{BS24}, \cite{ER93}, \cite{MOR16}, \cite{O03}, \cite{YO12}, \cite{PR90}, \cite{GT87}, \cite{WOY14}). 

The study of biharmonicity for isometric immersions, i.e. submanifolds, in Euclidean spheres has led to several classification theorems and important results (see, for example, \cite{FO22}, \cite{YO12} and \cite{OC20}). 

In \cite{LO07}, the authors studied diagonal biharmonic immersions $\phi:M\to\s^{n_1+n_2+1}$, $\phi\left (p\right) = \left(\alpha\phi_1, \beta\phi_2\right)$, where $\phi_1:\left(M^m, g\right)\to\s^{n_1}(r_1)$ and $\phi_2:\left(M^m, g\right)\to\s^{n_2}(r_2)$ are minimal immersions. Here $\alpha$ and $\beta$ are real numbers such that $\alpha^2r_1^2 + \beta^2r_2^2 = 1$, and by $\s^m\left (r\right)$ we indicate the $m$-dimensional Euclidean sphere of radius $r$. When $r = 1$, we write $\s^m$ instead of $\s^m(1)$. In particular, it was proved that any proper-biharmonic immersion $\phi$ with constant mean curvature from a $2$-dimensional sphere into $\s^n$ has to be the diagonal sum of two different Boruvka minimal immersions. We know that Boruvka minimal immersions are given by homogeneous polynomial maps.

Motivated by the above result, in our paper we study the diagonal maps from $\s^m$ into $\s^n$, $ \varphi = \i \circ (\varphi_1, \varphi_2) : \s^m \to \s^{n_1 + n_2 + 1}$, where $\i$ denotes the canonical inclusion of the standard product space $\s^{n_1} \times \s^{n_2}$ into $\s^{n_1 + n_2 + 1}$, and $ \varphi_1 : \s^m \to \s^{n_1}(r_1)$, $\varphi_2 : \s^m \to \s^{n_2}(r_2)$, $r_1^2 + r_2^2 = 1$, are homogeneous polynomial maps of degrees $k_1$ and $k_2$, respectively. 

Then, we provide a result for constructing proper biharmonic maps by using harmonic homogeneous polynomial maps between spheres, where these maps are defined on unit spheres of different dimensions. More precisely, we consider the case when $ \varphi_1 : \s^{m_1} \to \s^{n_1}(r_1)$, $\varphi_2 : \s^{m_2} \to \s^{n_2}(r_2)$, $r_1^2 + r_2^2 = 1$, are harmonic homogeneous polynomial maps of degrees $k_1$ and $k_2$, respectively. Then we study the biharmonicity of the product map $ \varphi = \i \circ (\varphi_1, \varphi_2) : \s^{m_1}\times\s^{m_2} \to \s^{n_1 + n_2 + 1}$, where $\i$ denotes the canonical inclusion of the standard product space $\s^{n_1} \times \s^{n_2}$ into $\s^{n_1 + n_2 + 1}$.

Our study could be seen as a continuation of previous work on biharmonic homogeneous polynomial maps between unit spheres (\cite{AOO23}, \cite{AO24}). Besides constructing proper biharmonic quadratic homogeneous polynomial maps, i.e. the degree $k=2$, the following rigidity result was proved: any proper biharmonic quadratic homogeneous polynomial map between spheres is obtained from a quadratic homogeneous polynomial map which lies, as a harmonic map, in the small hypersphere of radius $1/\sqrt{2}$ of the target sphere.

\textbf{Conventions.} We use the following sign conventions for the rough Laplacian, that acts on the set $C\left(\phi^{-1}TN\right)$ of all sections of the pull-back bundle $\phi^{-1}TN$, and for the curvature tensor field
$$
\Delta \sigma = -\textnormal{trace}_g\nabla^2\sigma, \quad R\left(X,Y\right)Z = \nabla_X\nabla_Y Z - \nabla_Y\nabla_YZ -\nabla_{\left[X,Y\right]}Z.
$$

\section{Preliminary results}

First, we recall the following results.

\begin{theorem}[\cite{LO07}]\label{th1}
Let M be a compact manifold and consider $\psi: M \rightarrow \s^n(r/\sqrt{2})$ a nonconstant map, where $\s^n(r/\sqrt{2})$ is a small hypersphere of radius $r/\sqrt{2}$ of $\s^{n+1}\left (r\right)$. The map $\phi = \i\circ\psi : M\rightarrow \s^{n+1}\left (r\right)$, where $\i$ is the canonical inclusion, is proper biharmonic if and only if $\psi$ is harmonic and the energy density $e(\psi)$ is constant.
\end{theorem}

We note that the above result  provides a method for constructing proper biharmonic maps starting from harmonic maps. In this article, we explore alternative methods to construct proper biharmonic maps, aiming to expand the available techniques for such constructions.

Now, let $\varphi:(M^m,g)\to\s^n\left (r\right)$ be a map and let $\i:\s^n\left (r\right)\to \r ^{n+1}$ be the standard isometric embedding of the Euclidean sphere of radius 
$r$ into the Euclidean space. We consider the composition map
$$
\Phi = \i \circ \varphi:(M^m,g)\to \r ^{n+1}.
$$
As usual, we identify locally $\d\varphi\left (X\right)$ with $\d\Phi\left (X\right)$, for any vector fields $X$ tangent to $M$.

We denote $\theta = \left\langle\d \varphi, \tau(\varphi)\right\rangle = \left\langle\d \Phi, \tau(\Phi)\right\rangle$, i.e. $\theta$ is the $1$-form on $M$ given by
$$
\theta\left (X\right) = \left\langle\d \varphi\left (X\right), \tau(\varphi)\right\rangle = \left\langle\d \Phi\left (X\right), \tau(\Phi)\right\rangle.
$$

We denote by $\theta^\sharp$ the vector field corresponding to the $1$-form $\theta$ via the musical isomorphism.

It is a well-known fact that the standard isometric embedding $\i : \s^n\left (r\right) \to \mathbb{R}^{n+1}$ is a totally umbilical hypersurface, whose normal vector field is given by $\overline r$. The vector field $\overline r$ associates to each point in $\s^n\left (r\right)$ its position vector.

Since $\nabla_U^{\r ^{n+1}}\overline r = U$, for any $U\in C(T\r ^{n+1})$, we get that the second fundamental form of $\s ^n\left (r\right)$ into $\r ^{m+1}$ is given by
$$
B(X,Y) = -\frac{1}{r^2}\left\langle X,Y \right\rangle  \overline r,
$$
for any $X,Y\in C(T\s ^n\left (r\right))$.

With these notations, we have
\begin{theorem}\label{th2}
Let $\varphi:(M^m,g)\to \s^n\left (r\right)$ be a map, $\i:\s^n\left (r\right)\to \r ^{n+1}$ be the standard isometric embedding and let $\Phi = \i \circ \varphi:(M^m,g)\to \r ^{n+1}$. Then, the bitension field of the $\varphi$ is given by 
\begin{align}\label{ec-1}
\tau_2(\varphi) =& \tau_2(\Phi) + \left( -\frac{1}{r^2}\Delta |\d  \Phi|^2 + \frac{2}{r^2}\textnormal{div}\theta^\sharp - \frac{1}{r^2}|\tau(\Phi)|^2 + \frac{2}{r^4}|\d \Phi|^4 \right ) \Phi \\
                 & + \frac{2}{r^2}|\d  \Phi|^2 \tau(\Phi)+ \frac{2}{r^2}\d \Phi\left(\textnormal{grad}|\d \Phi|^2 \right). \nonumber
\end{align}
\end{theorem}
We note that for the case $r=1$, this result can be found in \cite{AOO23}.

Further, consider the diagram below
\bigskip
\begin{center}
  \begin{tikzpicture}
  \matrix (m) [matrix of math nodes,row sep=3em,column sep=4em,minimum width=2em]
  {
     \r^{m+1} & \r^{n+1} \\
     \s^m & \s^n\left (r\right) \\};
  \path[-stealth]
    (m-2-1) edge node [left] {$\i$} (m-1-1)
    (m-1-1) edge node [above] {$F$} (m-1-2)
    (m-2-1) edge node [above] {$\Phi$} (m-1-2)
    (m-2-2) edge node [right] {$\i$} (m-1-2)
    (m-2-1) edge node [below] {$\varphi$} (m-2-2);
\end{tikzpicture}
\end{center}
where $F:\r ^{m+1}\rightarrow\r ^{n+1}$ is a vector valued function such that each component is a homogeneous polynomial of degree $k$. Such a map $F$ is called form of degree $k$. When $k=2$, we say that $F$ is a quadratic form, and we will keep the same terminology also for the induced map $\varphi$. We will always assume that $\varphi$ is not constant. Then, we have

\begin{theorem}{\cite{BW03}}\label{th3}
The tension field of the map $\varphi$ is given by
\begin{align}\label{equ10}
  \tau(\varphi) =& -\Dz F + \left(\frac{1}{r^2}\left|\dz F\right|^2 - k(m+2k-1)\right) \Phi.
\end{align}
\end{theorem}

For the special case of quadratic forms we have the following result.

\begin{proposition}{\cite{BW03}}\label{pr1}
Assume that $\varphi:\s^m\to\s^n$ is non constant quadratic form. Then the following relations are equivalent
\begin{itemize}
  \item[1)] $\tau(\varphi) = 0$,
  \item[2)] $\Dz F = 0$,
  \item[3)] $e(\varphi) = m+1$. 
\end{itemize}
\end{proposition}
\begin{remark}
We note that the proof of this result relies on the strict positivity of the energy density, i.e. $e(\varphi)>0$.
\end{remark}
Next, we recall several known results concerning harmonic forms of degree $k$.
\begin{proposition}{\cite{BW03}}\label{pr3}
Let $F:\r^{m+1}\rightarrow\r^{n+1}$ be a harmonic form of degree $k\in\mathbb{N}^*$. Suppose that $F$ restricts to the map $\varphi:\s^{m}\rightarrow\s^{n}$. Then $\varphi$ is harmonic with constant energy density $e(\varphi) = k(k + m -1)/2$, i.e. $\varphi$ is an eigenmap with $\nu = k(k + m -1)$.
\end{proposition}

\begin{proposition}{\cite{BW03}}\label{pr4}
Let $\varphi:\s^m\rightarrow\s^n$ be a harmonic with constant energy density $e(\varphi) = \alpha > 0$.
Then there exists a unique $k\in\mathbb{N}^*$ such that $\alpha = k(m+k-1)/2$ and there exists a unique vector valued function $F:\r^{m+1}\rightarrow\r^{n+1}$ such that each component is either a harmonic homogeneous polynomial of degree $k$, or the null polynomial, and $F$ restricts to $\varphi$.
\end{proposition}

\begin{remark}
In general, the condition $\tau(\varphi) = 0$ does not imply that $\Dz F = 0$. The following example illustrates precisely how $\tau(\varphi)$ vanishes, but $\Dz F \neq 0$.
\end{remark} 
\begin{example}
Let $\tilde F:\r^3\to\r^3$
\begin{equation*}
  \tilde F(x,y,z)=\left( x^2+y^2+z^2\right)^p \cdot \left(x,y,z\right), \quad p\in \mathbb{N}, \ p\geq 1.
\end{equation*}
Consider its restriction $\varphi:\s^2\to\s^2$. Then, by direct computation in $\r^{m+1}$, by restricting to $\s^2$ we obtain
\begin{align*}
\Dz \tilde F=&-2p(2p+3)\cdot \tilde \Phi, \\
\left|\dz \tilde F\right|^2 =& \left(4p^2+4p+3\right).
\end{align*}
Now, from Equation \eqref{equ10} it follows directly that $\tau(\f) = 0$.

We note that although $\varphi$ is harmonic and it has constant energy density, $\tilde F$ is not harmonic, i.e. $\Dz \tilde F \neq 0$.

However, there exists a unique harmonic homogeneous polynomial map $F:\r^3\to\r^3$, $F(x,y,z) = (x,y,z)$ such that its restriction is $\varphi$, and 
$$
\tilde F(x,y,z) = \left( x^2+y^2+z^2\right)^p F(x,y,z).
$$
\end{example}

Now, in the set of the forms of degree $k$, $k\geq 1$, that take $\s^m$ to $\s^n$ we define the following relation 
$$
F \sim G \textnormal{  if and only if  } F_{\big | \s^m} = G_{\big | \s^m}.
$$
It is easy to observe that $\sim$ is an equivalence relation. We factorize the set of the forms of degree $k$ that take $\s^m$ to $\s^n$ to this relation and we obtain equivalence classes denoted $[F]$. 

We note that if $[F] = [G]$, then there exists an integer $p\geq 1$ such that we have on $\r^{m+1}$ either $F(\overline x) = |\overline x|^{2p}G(\overline x)$ or $G(\overline x) = |\overline x|^{2p}F(\overline x)$. Also, as no nonzero polynomial multiple of $|\overline x|^2$ is harmonic (see \cite{ABR13}), it follows that any class $[F]$ contains at most one harmonic form of degree $k$, that is of minimal degree in its class.
 
\begin{proposition}\label{pr5}
Let $F:\r^{m+1}\rightarrow\r^{n+1}$ be a form of degree $k$, of minimal degree in its class $[F]$, such that its restriction $\varphi:\s^m\rightarrow\s^n$ is harmonic. Then $\varphi$ has constant energy density $e(\varphi) = k(k + m -1)/2$ and $\Dz F = 0$.
\end{proposition}

\begin{proof}
It is known that (one may see, for example, \cite{AOO23})
$$
\left|\d\varphi\right|^2 = \left|\d\Phi\right|^2 = \left|\dz F\right|^2- k^2.
$$
Therefore, it follows that $\left|\dz F\right|^2$ is constant on $\s^m$  if and only if $\varphi$ has constant energy density. 

From Equation \eqref{equ10} it follows that on $\s^m$ we have
\begin{equation}\label{e2}
\Dz F = \left(\left|\dz F\right|^2-k(m+2k-1)\right) \Phi.
\end{equation}

If $\left|\dz F\right|^2-k(m+2k-1)=0$, then the conclusion follows immediately.

If $\left|\dz F\right|^2-k(m+2k-1)=c$, where $c$ is non-zero real constant, then on $\s^m$ we have
$$
\Dz F = c \cdot \Phi, 
$$
and further on $\r^{m+1}$ we have 
$$
\left|\overline x\right|^2\Dz F = c F,
$$
but this is a contradiction to the hypothesis that $F$ is of minimal degree in its class.

Now suppose that $\f$ does not have constant energy density, therefore $\left|\dz F\right|^2-k(m+2k-1)$ is not constant on $\s^m$.
By homogenizing Equation \eqref{e2} we obtain
\begin{equation}\label{e3}
\left|\overline x\right|^{2k}\Dz F = \left(\left|\dz F\right|^2-k(m+2k-1)\left|\overline x\right|^{2k-2}\right) F,
\end{equation}
that holds on $\r^{m+1}$.

As $F$ is a homogeneous polynomial vector function of degree $k$, $\Dz F$ is a homogeneous polynomial vector function of degree $k-2$,  $\left|\dz F\right|^2 -k(m+2k-1)\left|\overline x\right|^{2k-2}$ is a homogeneous polynomial function of degree $2k-2$, from Equation \eqref{e3} it follows that $\left|\overline x\right|^2$ divides $F$ because the largest power of $\left|\overline x\right|^2$ that can divide $\left|\dz F\right|^2 -k(m+2k-1)\left|\overline x\right|^{2k-2}$ is $\left|\overline x\right|^{2k-2}$. But this is a contradiction to the hypothesis that $F$ is minimal degree in its class.  Therefore, our assumption that and the energy density of $\varphi$ is not constant, is false.
\end{proof}
\begin{remark}
For more details on homogeneous polynomial maps and their decomposition one may see \cite{ABR13}.
\end{remark}
\begin{proposition}\label{pr2}
Let $F:\r^{m+1}\rightarrow\r^{n+1}$ be a quadratic form, such that it restricts to $\varphi:\s^m\rightarrow\s^n\left (r\right)$. Then, on $\s^m$ we have
\begin{equation*}
-2 \Dz\left(\left|\dz F\right|^2\right) - 2\left|\nz\dz F\right|^2 + \left|\Dz F\right|^2 = 4r^2(m+1)(m+3).
\end{equation*}
\end{proposition}

Following the computations from \cite{AOO23}, we obtain
\begin{theorem}\label{th4}
Let $F:\r^{m+1}\rightarrow\r^{n+1}$ be a form of degree $k$ such that it restricts to $\varphi:\s^m\rightarrow\s^n\left (r\right)$. The bitension field of the map $\varphi$ is given by
\begin{align}\label{ec-2}
\tau_2(\varphi) = & \Dz\Dz F + 2\left(mk+2k^2-3k-m+3-\frac{1}{r^2}\left|\dz F\right|^2\right)\Dz F \nonumber \\
               & + \frac{1}{r^2}\left( - 2\Dz\left(\left|\dz F\right|^2\right) - 2\left|\nz\dz F\right|^2 + \left|\Dz F\right|^2\right.\\
               & \left. - 2\left(2mk+6k^2-6k-m+3\right)\left|\dz F\right|^2 + \frac{2}{r^2}\left|\dz F\right|^4 + 4r^2k^2\left(m+2k-1\right)\right) \Phi \nonumber \\
               & + \frac{2}{r^2}\dz F\left(\stackrel{o}{\textnormal{grad}}\left(\left|\dz F\right|^2\right)\right), \nonumber
\end{align}
where $\dz$, $\nz$, $\Dz$ and $\stackrel{o}{\textnormal{grad}}$ denote operators that act on $\r^{m+1}$.
\end{theorem}
We note that for the case when the radius $r=1$, the last two result can also be found in \cite{AOO23}.

\begin{proposition}\label{pr6}
Let $F:\r^{m+1}\to\r^{n+1}$ be a harmonic form of degree $k$ such that it restricts to a map $\varphi:\s^m\to\s^n\left (r\right)$. Then, on $\r^{m+1}$ we have
\begin{align}\label{e4}
\Dz\left(\left|\dz F\right|^2\right) & =- 2r^2k(k-1)(m+2k-1)(m+2k-3),\\
\left|\nz \dz F\right|^2 &=r^2k(k-1)\left(m^2-4m+3+4k(m-2)+4k^2\right).\nonumber
\end{align}
\end{proposition}

\begin{proof}
If $F$ is harmonic, then from Proposition \eqref{pr3} it follows that $\varphi$ is harmonic and on $\s^m$ we have
$$
\left|\dz F\right|^2=kr^2(m+2k-1),
$$
therefore on $\r^{m+1}$ we have 
$$
\left|\dz F\right|^2=kr^2(m+2k-1)\left|\overline x\right|^{2k-2}.
$$
Then, by direct computation we obtain on $\s^m$
\begin{equation}\label{e5}
\Dz\left(\left|\dz F\right|^2\right)  =- 2r^2k(k-1)(m+2k-1)(m+2k-3).
\end{equation}
As $F$ is harmonic, $\varphi$ is harmonic, thus biharmonic, and from Equation \eqref{ec-2} we obtain by direct computations on $\s^m$
$$
\left|\nz \dz F\right|^2 =r^2k(k-1)\left(m^2-4m+3+4k(m-2)+4k^2\right).
$$
\end{proof}

\section{New construction method of proper biharmonic maps}

We begin this section with an application to Theorem \eqref{th4}.
\begin{theorem}\label{th12}
Let $G:\r^{m+1}\to\r^{n+1}$ be a harmonic form of degree $k$ such that its restriction $\psi:\s^m\to\s^n$ is not constant. Let $F:\r^{m+1}\to\r^{(m+1)(n+1)}$ be a form of degree $k+1$ defined by
\begin{equation*}
F(\overline x) = \left(x^1G(\overline x), x^2G(\overline x), \ldots, x^{m+1}G(\overline x)\right).
\end{equation*}
Then, its restriction $\varphi:\s^m\to\s^{(m+1)(n+1)-1}$ is proper biharmonic if and only if $m=1$.
\end{theorem}

\begin{proof}
We aim to compute the terms from the right hand side of Equation \eqref{ec-2} in terms of $G$.

Let $\i:\s^{(m+1)(n+1)-1}\to \r ^{(m+1)(n+1)}$ be the standard isometric embedding of the unit Euclidean sphere into Euclidean space and consider the composition map
$$
\Phi = \i \circ \varphi:\s^{(m+1)(n+1)-1}\to \r ^{(m+1)(n+1)}.
$$

As the vector function $G$ is harmonic, i.e. $\Dz G = 0$, it follows from Theorem \eqref{th3} that $\psi$ is harmonic and 
\begin{equation*}
\left|\dz G\right|^2 = k(m+2k-1), \quad \textnormal{ on } \s^m,
\end{equation*} 
and further
\begin{equation}\label{ec-3}
\left|\dz G\right|^2 = k(m+2k-1)\left|\overline x\right|^{2(k-1)}, \quad \textnormal{ on } \r^{m+1}.
\end{equation} 
Next, we compute $\left|\dz F\right|^2$  and $\Dz F$. On $\r^{m+1}$ we have
\begin{align}\label{ec-3.1}
\frac{\partial F}{\partial x^1}\left(\overline x\right) &= \left(G\left(\overline x\right) + x^1 \frac{\partial G}{\partial x^1}\left(\overline x\right), x^2 \frac{\partial G}{\partial x^1}\left(\overline x\right), \ldots, x^{m+1}\frac{\partial G}{\partial x^1}\left(\overline x\right)\right) \nonumber\\

\frac{\partial F}{\partial x^2}\left(\overline x\right) &= \left( x^1 \frac{\partial G}{\partial x^2}\left(\overline x\right), G\left(\overline x\right) + x^2 \frac{\partial G}{\partial x^2}\left(\overline x\right), \ldots, x^{m+1}\frac{\partial G}{\partial x^2}\left(\overline x\right)\right) \\
\ldots & \nonumber\\
\frac{\partial F}{\partial x^{m+1}}\left(\overline x\right) &= \left( x^1 \frac{\partial G}{\partial x^{m+1}}\left(\overline x\right),  x^2 \frac{\partial G}{\partial x^{m+1}}\left(\overline x\right), \ldots, G\left(\overline x\right) + x^{m+1}\frac{\partial G}{\partial x^{m+1}}\left(\overline x\right)\right).\nonumber
\end{align}
Then, on $\r^{m+1}$ we have
\begin{align*}
\left|\dz F\right|^2_{\overline x} =& (m+1) \left|G\left(\overline x\right)\right|^2 + \left|\overline x\right|^2 \left|\dz G\right|^2_{\overline x} + 2\left\langle G\left(\overline x\right), \dz G_{\overline x}\left(\overline x\right)\right\rangle \\
=& (m+1) \left|G\left(\overline x\right)\right|^2 + \left|\overline x\right|^2 \left|\dz G\right|^2_{\overline x} + 2\left\langle G\left(\overline x\right), k G\left(\overline x\right)\right\rangle \\
=& (m+2k+1) \left|G\left(\overline x\right)\right|^2 + \left|\overline x\right|^2 \left|\dz G\right|^2_{\overline x}.
\end{align*}
Therefore, as $G$ is a $k$-form, i.e. $\left|G\left(\overline x\right)\right|^2 = \left|\overline x\right|^{2k}$, and taking into account Equation \eqref{ec-3} we obtain on $\r^{m+1}$
\begin{equation}\label{ec-4}
\left|\dz F\right|^2_{\overline x} = \left(m+2k+1+k(m+2k-1)\right)\left|\overline x\right|^{2k}.
\end{equation}
Next, from Equations \eqref{ec-3.1} we obtain on $\r^{m+1}$
\begin{align}\label{ec-4.1}
\frac{\partial F}{\partial \left(x^1\right)^2}\left(\overline x\right) &= \left(2\frac{\partial G}{\partial x^1}\left(\overline x\right) + x^1 \frac{\partial G}{\partial \left(x^1\right)^2}\left(\overline x\right), x^2 \frac{\partial G}{\partial \left(x^1\right)^2}\left(\overline x\right), \ldots, x^{m+1}\frac{\partial G}{\partial \left(x^1\right)^2}\left(\overline x\right)\right) \nonumber\\

\frac{\partial F}{\partial \left(x^2\right)^2}\left(\overline x\right) &= \left( x^1 \frac{\partial G}{\partial \left(x^2\right)^2}\left(\overline x\right), 2\frac{\partial G}{\partial x^2}\left(\overline x\right) + x^2 \frac{\partial G}{\partial \left(x^2\right)^2}\left(\overline x\right), \ldots, x^{m+1}\frac{\partial G}{\partial \left(x^2\right)^2}\left(\overline x\right)\right) \\

\ldots &\nonumber \\

\frac{\partial F}{\partial \left(x^{m+1}\right)^2}\left(\overline x\right) &= \left( x^1 \frac{\partial G}{\partial \left(x^{m+1}\right)^2}\left(\overline x\right),  x^2 \frac{\partial G}{\partial \left(x^{m+1}\right)^2}\left(\overline x\right), \ldots, 2\frac{\partial G}{\partial x^{m+1}}\left(\overline x\right) + x^{m+1} \frac{\partial G}{\partial \left(x^{m+1}\right)^2}\left(\overline x\right)\right).\nonumber
\end{align}
As $G$ is harmonic, by summing Equations \eqref{ec-4.1} we obtain on $\r^{m+1}$ 
\begin{equation}\label{ec-5}
\Dz F  = -2\left(\frac{\partial G}{\partial x^1}\left(\overline x\right), \frac{\partial G}{\partial x^2}\left(\overline x\right), \ldots, \frac{\partial G}{\partial x^{m+1}}\left(\overline x\right)\right), 
\end{equation}
and further, as $G$ is harmonic, 
\begin{equation}\label{ec-6}
\Dz \Dz F = 0. 
\end{equation}
From Equation \eqref{ec-4} it follows by direct computation that on $\s^m$ we have
\begin{equation}\label{ec-7}
\Dz \left(\left|\dz F\right|^2\right) = -2k(m+2k-1)\left(m+2k+1+k(m+2k-1)\right).
\end{equation}
Now, we compute $\dz F\left(\stackrel{o}{\textnormal{grad}}\left(\left|\dz F\right|^2\right)\right)$. From Equation \eqref{ec-4} it follows that on $\r^{m+1}$ we have
\begin{equation*}
\stackrel{o}{\textnormal{grad}}\left(\left|\dz F\right|^2\right) = \left(m+2k+1+k(m+2k-1)\right) 2k\left|\overline x\right|^{2(k-1)}\overline x.
\end{equation*}
Thus, on $\s^m$ we have
\begin{align}\label{ec-8}
\dz F\left(\stackrel{o}{\textnormal{grad}}\left(\left|\dz F\right|^2\right)\right) = & 2k\left(m+2k+1+k(m+2k-1)\right) \dz F_{\overline x}\left(\overline x\right) \nonumber \\
& = 2k(k+1)\left(m+2k+1+k(m+2k-1)\right)\Phi.
\end{align}
Now, we replace Equations \eqref{ec-4}, ... , \eqref{ec-8} in Equation \eqref{ec-2} and we obtain by direct calculations 
\begin{equation}\label{ec-9}
\tau_2(\varphi) = 2(1-m)\Dz F + \left\{-2\left|\nz\dz F\right|^2 +4\left(1+2k-2k^3-k^2(m-3)+3\right) \right\} \Phi.
\end{equation}

If $m=1$, Equation \eqref{ec-9} becomes 
$$
\tau_2(\varphi) = \left\{-2\left|\nz\dz F\right|^2 +8\left(1+k+k^2-k^3\right) \right\} \Phi.
$$
Thus, 
$$
\left|\nz\dz F\right|^2 = 4\left(1+k+k^2-k^3\right)
$$
and
$$
\tau_2(\varphi) = 0.
$$

If $\tau_2(\varphi) = 0$, then from Equation \eqref{ec-9} on $\s^m$ we have
\begin{equation}\label{ec-10}
2(1-m)\Dz F = - \left\{-2\left|\nz\dz F\right|^2 +4\left(1+2k-2k^3-k^2(m-3)+3\right) \right\} \Phi.
\end{equation}
It follows that on $\s^m$
$$
4(1-m)^2\left|\Dz F\right|^2 = \left(-2\left|\nz\dz F\right|^2 +4\left(1+2k-2k^3-k^2(m-3)+3\right) \right)^2,
$$
and further, using Equations \eqref{ec-4} and \eqref{ec-5}, we obtain on $\s^m$
$$
16k(m+2k-1)(1-m)^2 = \left(-2\left|\nz\dz F\right|^2 +4\left(1+2k-2k^3-k^2(m-3)+3\right) \right)^2.
$$
Therefore, the coefficient of $\Phi$ from Equation \eqref{ec-10} becomes either 
$$
4(m-1)\sqrt{k(m+2k-1)} \quad \textnormal{ or } \quad -4(m-1)\sqrt{k(m+2k-1)}.
$$ 

First, we assume that the coefficient of $\Phi$ is $4(m-1)\sqrt{k(m+2k-1)}$. Then Equation \eqref{ec-10}becomes
\begin{equation*}
2(m-1)\Dz F =  4(m-1)\sqrt{k(m+2k-1)} \Phi.
\end{equation*} 
We assume $m\neq 1$. Then, on $\s^{(m+1)(n-1)}$ we have $\Dz F =  2\sqrt{k(m+2k-1)} \Phi$. This is equivalent to
\begin{align*}
-\frac{\partial G}{\partial x^1}\left(\overline x\right) &= \sqrt{k(m+2k-1)} x^1 G\left(\overline x\right), \\
-\frac{\partial G}{\partial x^2}\left(\overline x\right) &= \sqrt{k(m+2k-1)} x^2 G\left(\overline x\right), \\
\ldots& \\
-\frac{\partial G}{\partial x^{m+1}}\left(\overline x\right) &= \sqrt{k(m+2k-1)} x^{m+1} G\left(\overline x\right). \\
\end{align*}
It follows that 
\begin{align*}
-x^1\frac{\partial G}{\partial x^1}\left(\overline x\right) &= \sqrt{k(m+2k-1)} \left(x^1\right)^2 G\left(\overline x\right), \\
-x^2\frac{\partial G}{\partial x^2}\left(\overline x\right) &= \sqrt{k(m+2k-1)} \left(x^2\right)^2 G\left(\overline x\right), \\
\ldots & \\
-x^{m+1}\frac{\partial G}{\partial x^{m+1}}\left(\overline x\right) &= \sqrt{k(m+2k-1)} \left(x^{m+1}\right)^2 G\left(\overline x\right). \\
\end{align*}
and by summing the above equations we obtain on $\s^m$
\begin{equation*}
-\dz G_{\overline x}\left(\overline x\right) = \sqrt{k(m+2k-1)} G\left(\overline x\right),
\end{equation*}
that is further equivalent to 
\begin{equation*}
-k G\left(\overline x\right) = \sqrt{k(m+2k-1)} G\left(\overline x\right),
\end{equation*}
that is false.

Now, we assume that the coefficient of $\Phi$ is $4(m-1)\sqrt{k(m+2k-1)}$. Similar, we obtain on $\s^m$
\begin{equation*}
k G\left(\overline x\right) = \sqrt{k(m+2k-1)} G\left(\overline x\right),
\end{equation*}
that is also false.

Therefore, the supposition $m\neq 1$ leads only to false results. In conclusion, $\tau_2(\varphi) = 0$ implies $m=1$.

Now, for $m=1$ from Equation \eqref{ec-9} it follows directly that $\tau_2(\varphi) = 0$. Concerning the harmonicity of $\varphi$, we replace Equations \eqref{ec-4} and \eqref{ec-5} in Equation \eqref{equ10} and we obtain
\begin{align*}
\tau(\varphi) =& -\Dz F -2k \Phi\\
              =& 2\left(\frac{\partial G}{\partial x^1}, \frac{\partial G}{\partial x^2}, \ldots, \frac{\partial G}{\partial x^{m+1}}\right) -2k\left(x^1 G, x^2 G, \ldots, x^{m+1} G\right).
\end{align*}
If we assume that $\tau(\varphi) = 0$, from the last equation we obtain on $\s^m$
$$
\left|\dz G\right|^2 = k^2,
$$
that is a contradiction to Equation \eqref{ec-3}.

In conclusion, the map $\varphi$ is proper biharmonic if and only if $m=1$.
\end{proof}

\begin{example}\label{ex1} To confirm the validity of Theorem \eqref{th12}, we construct and study the biharmonicity of a special class of closed curves in $\s^3$. While the construction satisfies the hypotheses of the theorem, we independently verify its biharmonicity by computing the tension and bitension fields. 

We consider a harmonic form of degree $k$ 
$$
G:\r^2\to\r^2.
$$
As we already know,  an orthogonal basis, with respect to the usual product, for the linear space of homogeneous harmonic polynomials of degree $k$ in 2 variables is given in polar coordinates by $r^k\cos(k\theta)$ and $r^k\sin(k\theta)$ (see, for example, \cite{ABR13}). Therefore, taking a complex number $z=x+i y$, an orthogonal basis, with respect to the usual product, for the linear space of homogeneous harmonic polynomials of degree $k$ in 2 variables is given by \begin{align*}
P_k(x,y) &= \textnormal{Re}\left(z^k\right),\\
Q_k(x,y) &= \textnormal{Im}\left(z^k\right).
\end{align*}
Therefore, without loss of generality we can assume that 
$$
G(x,y) = \left( P_k(x,y), Q_k(x,y)\right).
$$
Now we define 
$$
F:\r^2\to\r^4, \quad F(x,y)=\left(xP_k(x,y), xQ_k(x,y), yP_k(x,y),y Q_k(x,y)\right),
$$
and we consider its restriction $\varphi:\s^1\to\s^3$.
For this particular case, by direct computations we obtain on $\r^2$
\begin{align}\label{ec-11}
\left|\dz G\right|^2 =& \ 2k^2\left(x^2+y^2\right)^{k-1},\\
\left|\dz F\right|^2 =& \ 2\left(k^2+k+1\right)\left(x^2+y^2\right)^{k},\nonumber
\end{align}
and therefore $\varphi$ is parameterised by arc length. Further on $\r^2$ we have
\begin{align}\label{ec-11.1}
\Dz F =& -2\left(\frac{\partial G}{\partial x}, \frac{\partial G}{\partial y}\right),\nonumber\\
\left|\Dz F\right|^2 =& \ 8k^2 \left(x^2+y^2\right)^{k-1},\\
\Dz \Dz F =& \ 0, \nonumber\\
\Dz\left(\left|\dz F\right|^2\right) = & -8k^2\left(k^2+k+1\right)\left(x^2+y^2\right)^{k-1},\nonumber 
\end{align}
and on $\s^1$
\begin{align*}
\left|\nz \dz F\right|^2 = & 8k^2 + 4k^4,\\ 
\dz F\left(\stackrel{o}{\textnormal{grad}}\left(\left|\dz F\right|^2\right)\right) = &  4k(k+1)\left(1+k+k^2\right)\Phi,
\end{align*}
that checks with the general case computed in the proof of the previous theorem.

Now, replacing in Equation \eqref{equ10} and \eqref{ec-1}, we obtain 
\begin{align*}
\tau(\varphi) =& -\Dz F - 2k \Phi,\\
\tau_2(\varphi) =& 0. 
\end{align*}
If $\tau(\f)  = 0$, then on $\s^1$ we have $-\Dz F = 2k F$, and further $\left|\Dz F\right| = 4k^2$ that is a contradiction to the second equation from System \eqref{ec-11.1}.
Therefore we obtain that $\varphi$ is proper biharmonic.
\end{example}

Moreover, it is not difficult to see that by acting on $F$ with the isometry
\begin{align*}
T = \frac{1}{\sqrt{2}}\begin{bmatrix}
             1 & 0 & 0 & 1 \\
             0 & 1 & 1 & 0 \\
             0 & 1 & -1 & 0 \\
             1 & 0 & 0 & -1
           \end{bmatrix},
\end{align*}
we obtain a new vector function $\tilde F = \left(\tilde F_1, \tilde F_2, \tilde F_3, \tilde F_4\right)$ such that
\begin{align*}
\tilde F_1 + i \tilde F_3 =& \frac{1}{\sqrt{2}}\overline z \cdot z^k = \frac{1}{\sqrt{2}}\left|z\right|^2\cdot z^{k-1}, \\
\tilde F_4 + i \tilde F_2 =& \frac{1}{\sqrt{2}}z \cdot z^k = \frac{1}{\sqrt{2}}z^{k+1}.
\end{align*}
Therefore, after rearranging the components of $\tilde F$, it takes the form
$$
\tilde F(z) = \frac{1}{\sqrt{2}}\left(|z|^2 z^{k-1}, z^{k+1}\right).
$$
On $S^1$, 
$$
\tilde \Phi(\cos t, \sin t) = \left(\cos \left(2(k-1)t\right),\sin \left(2(k-1)t\right),\cos \left(2(k+1)t\right),\sin \left(2(k+1)t\right)\right).
$$
\begin{remark}
The previous examples is part of a larger family of biharmonic curves in $\s^3$, first time discovered in \cite{CMO01} where the biharmonic curves in $\s^3$ were classified. 

\end{remark}

We note that since the degree $k$ can be odd, it follows that in these cases there is no constant component of the map $\varphi$, therefore the image of $\varphi$ does not lie in a hyperplane as in the case of the maps constructed using Theorem \eqref{th1}.

\bigskip

Next, as an application to Theorem \eqref{th2} we consider the special case when 
\begin{equation}\label{ec-12}
\varphi = \i \circ \left(\varphi_1,\varphi_2\right):\s^m\to\s^n,
\end{equation}
where $n = n_1+n_2+1$, $\i$ is the canonical inclusion of the standard product $\s^{n_1}(r_1)\times\s^{n_2}(r_2)$ in $\s^n$,  and $\varphi_1$ and $\varphi_2$ are given as in the below diagrams

\bigskip
\begin{center}
  \begin{tikzpicture}
  \matrix (m) [matrix of math nodes,row sep=3em,column sep=4em,minimum width=2em]
  {
     \r^{m+1} & \r^{n_1+1} \\
     \s^m & \s^{n_1}(r_1) \\};
  \path[-stealth]
    (m-2-1) edge node [left] {$\i$} (m-1-1)
    (m-1-1) edge node [above] {$F_1$} (m-1-2)
    (m-2-1) edge node [above] {$\Phi_1$} (m-1-2)
    (m-2-2) edge node [right] {$\i_1$} (m-1-2)
    (m-2-1) edge node [below] {$\varphi_1$} (m-2-2);
\end{tikzpicture} 
\begin{tikzpicture}
  \matrix (m) [matrix of math nodes,row sep=3em,column sep=4em,minimum width=2em]
  {
     \r^{m+1} & \r^{n_2+1} \\
     \s^m & \s^{n_2}(r_2) \\};
  \path[-stealth]
    (m-2-1) edge node [left] {$\i$} (m-1-1)
    (m-1-1) edge node [above] {$F_2$} (m-1-2)
    (m-2-1) edge node [above] {$\Phi_2$} (m-1-2)
    (m-2-2) edge node [right] {$\i_2$} (m-1-2)
    (m-2-1) edge node [below] {$\varphi_2$} (m-2-2);
\end{tikzpicture}
\end{center}
where $r_1^2 + r_2^2 = 1$ and $F_1$ and $F_2$ are forms of degree $k_1$, respectively $k_2$, i.e. on $\r^{m+1}$ we have $\left|F_1\left(\overline{x}\right)\right|^2 = r_1^2\left|\overline x\right|^{2k_1}$, respectively $\left|F_2\left(\overline{x}\right)\right|^2 = r_2^2\left|\overline x\right|^{2k_2}$.
Then, for $\varphi$ we have the following diagram
\bigskip
\begin{center}
  \begin{tikzpicture}
  \matrix (m) [matrix of math nodes,row sep=6em,column sep=7em,minimum width=4em]
  {
     \r^{m+1} & \r^{n_1+n_2+2} \\
     \s^m & \s^{n_1+n_2+1} \\};
  \path[-stealth]
    (m-2-1) edge node [left] {$\i$} (m-1-1)
    (m-1-1) edge node [above] {$F=(F_1,F_2)$} (m-1-2)
    (m-2-1) edge node [above] {$\Phi=(\Phi_1,\Phi_2)$} (m-1-2)
    (m-2-2) edge node [right] {$\i$} (m-1-2)
    (m-2-1) edge node [below] {$\varphi=\i\circ(\varphi_1,\varphi_2)$} (m-2-2);
\end{tikzpicture} 
\end{center}

\begin{theorem}\label{th6}
The tension field of the map $\varphi$ given in Equation \eqref{ec-12} is given by
\begin{align}\label{ec-13}
\tau(\varphi) =& -\Dz F + \left(\left(\left|\dz F\right|^2 - k_1^2r_1^2-k_2^2r_2^2+k_1(1-m-k_1)\right)\Phi_1,\right. \\
               &\left.  \left(\left|\dz F\right|^2 - k_1^2r_1^2-k_2^2r_2^2+k_2(1-m-k_2)\right)\Phi_2\right).\nonumber
\end{align}
\end{theorem}
\begin{proof}
It is well-known that the standard isometric embedding $\i:\s^n\to \r ^{n+1}$ is a totally umbilical hypersurface with the unit normal vector field $\overline r$, that  associates to any point the corresponding position vector.

For simplicity of notation, an arbitrary point $p\in\s^m$ will be denoted by $\overline x$. Now let $\overline x\in\s^m$ be an arbitrary point and consider a geodesic frame field $\{X_i\}_{i=1}^m$ around $\overline x$, defined on the open subset $U$ of $\s^m$, $\overline x\in U$.

On U, as in \cite{AOO23}, we have
\begin{equation}\label{ec-15}
|\d \Phi|^2=\sum_{i=1}^{m}\left|\d \Phi(X_i)\right|^2 = \sum_{i=1}^{m}\left|\dz F(X_i)\right|^2 = \left|\dz F\right|^2-\left|\dz F(\overline r)\right|^2 = \left|\dz F\right|^2-|\overline  r F|^2
\end{equation}
Then, at $\overline x$ we have
\begin{equation*}
|\d \Phi|^2_{\overline x} = \left|\dz F\right|^2_{\overline x} - \left|(F\circ\gamma)'(1)\right|^2,
\end{equation*}
where $\gamma(t)= t\overline x = t \overline r(\overline x)$.

On $U$ we have
\begin{align}\label{ec-16}
\tau(\Phi) &= \dz F(\tau(i)) + \textnormal{trace}\nz\dz F(\d \i\cdot,\d \i\cdot) \nonumber\\
           &= \dz F(-m \overline r)+\tz(F)-(\nz\dz F)(\overline r,\overline r) \\
           &= - m \overline r F + \tz(F) -  \overline r(\overline r F)+ r F \nonumber\\
           &= \tz(F)+(1-m) \overline r F- \overline r(\overline r F).\nonumber
\end{align}
Equivalently, on $U$
\begin{equation*}
\Delta\Phi = \Dz F - (1-m)\overline r F + \overline r(\overline r F).
\end{equation*}

Since $F_1$ is a form of degree $k_1$ and $F_2$ is a form of degree $k_2$, we have
 \begin{align}\label{ec-17}
   \left(\overline rF\right)(\overline{x}) & = \overline r(\overline{x})F \nonumber\\
    & = \frac{\textnormal{d}}{\textnormal{d}t}\Big|_{t=1}\left\{F(t\overline{x})\right\}\\
    & = \frac{\textnormal{d}}{\textnormal{d}t}\Big|_{t=1}\left\{\left(t^{k_1}F_1(\overline{x}),t^{k_2}F_2(\overline{x})\right)\right\} \nonumber\\
    & = (k_1 F_1(\overline{x}), k_2 F_2(\overline{x})).\nonumber
 \end{align}
 This kind of formula holds for the action of $\overline r$ on $\Dz F$, as
\begin{equation*}
  \left( \Dz F \right)(t\overline{x})=\left( t^{k_1 - 2}\left( \Dz F_1\right)(\overline{x}), t^{k_2 - 2} \left(\Dz F_2\right)(\overline{x}) \right).
\end{equation*}
Thus,
\begin{equation}\label{ec-18}
  \overline r\left( \Dz F \right) = \left( (k_1 - 2) \Dz F_1,(k_2-2)\Dz F_2 \right).
\end{equation}

Using equation \eqref{ec-17}, on $\s^m$ Equation \eqref{ec-15} becomes
\begin{align}\label{ec-21}
\left| \textnormal{d}\Phi\right|^2 & = \left| \dz F\right|^2 - \left|(k_1 F_1,k_2F_2)\right|^2 \\
                                   & = \left| \dz F\right|^2 - \left(k_1 ^2 r_1^2 + k_2^2 r_2^2\right). \nonumber
\end{align}

Using Equations \eqref{ec-16} and \eqref{ec-17}, it follows that
\begin{align}\label{ec-19}
  \tau\left(\Phi\right) =& \tz\left(F\right) - m (k_1F_1,k_2F_2) - \left( k_1(k_1 - 1)F_1, k_2 (k_2-1)F_2\right)\\
                        =& \tz\left(F\right) +(1 - m) (k_1\Phi_1,k_2\Phi_2) - ( k_1 ^{2}\Phi_1, k_2 ^{2} \Phi_2).\nonumber
\end{align}
Thus, on $\s^m$ we have
\begin{align}\label{ec-20}
  \Delta \Phi =& \Dz F + \left( k_1\left(m + k_1 - 1\right)\Phi_1, k_2 \left(m + k_2-1\right)\Phi_2\right)\\
              =& \left(\Dz F_1 + k_1(m+k_1-1)\Phi_1,\Dz  F_2 + k_2 (m+k_2-1)\Phi_2 \right).\nonumber
\end{align}

Since
\begin{align*}
\tau(\varphi)  = & \tau(\Phi) + \left|\d \Phi\right|^2\Phi, 
\end{align*}
using Equations \eqref{ec-19} and \eqref{ec-21} it follows that 
\begin{align*}
\tau(\varphi)  = & \tau(\Phi) + \left(\left|\dz F\right|^2 -k_1^2r_1^2-k_2^2r_2^2\right)\Phi \\
               = &\left(-\Dz F_1 + k_1(1-m-k_1)\Phi_1, -\Dz F_2 + k_2(1-m-k_2)\Phi_2\right)\\
               & +  \left(\left|\dz F\right|^2 -k_1^2r_1^2-k2^2r_2^2\right)\Phi \\  
               = & -\Dz F + \left(\left(\left|\dz F\right|^2 - k_1^2r_1^2-k_2^2r_2^2+k_1(1-m-k_1)\right)\Phi_1,\right. \\
               &\left.  \left(\left|\dz F\right|^2 - k_1^2r_1^2-k_2^2r_2^2+k_2(1-m-k_2)\right)\Phi_2\right).
\end{align*}
\end{proof}

\begin{corollary}\label{cor1}
If $\Dz F_1 = 0$ and $\Dz F_2 = 0$, then the map $\varphi$ given in Equation \eqref{ec-12} is harmonic if and only if $k_1 = k_2$.
\end{corollary}

\begin{proof}
If $\Dz F = 0$, then $\Dz F_1 = 0$ and $\Dz F_2 = 0$. It follows from Theorem \eqref{th3} that on $\s^m$ we have $\left|\dz F_1\right|^2 = k_1r_1^2(m+2k_1-1)$ and $\left|\dz F_2\right|^2 = k_2r_2^2(m+2k_2-1)$. Therefore, as $r_2^2 = 1 - r^2_1$, from Equation \eqref{ec-13} it follows by direct calculations that
\begin{align*}
\tau(\varphi) =& (1-r_1^2)\left(k_2-k_1\right)\left(\left(m+k_2+k_1+1\right)\Phi_1, -\left(m+k_2+k_1+1\right)\Phi_2\right).
\end{align*}
Thus, $\varphi$ is harmonic if and only if $k_1 = k_2$.
\end{proof}

\begin{corollary}\label{cor3}
If $F_1:\r^{m+1}\to\r^{n_1+1}$ and $F_2:\r^{m+1}\to\r^{n_2+1}$ are forms of degree $k_1$, respectively $k_2$, each of them of minimal degree in its class, such that their restrictions $\varphi_1:\s^m\to\s^{n_1}(r_1)$ and $\varphi_2:\s^m\to\s^{n_2}(r_2)$ are harmonic, then the map $\f$ given in Equation \eqref{ec-12} is harmonic if and only if $k_1 = k_2$. 
\end{corollary}
\begin{proof}
Using Proposition \eqref{pr5} it follows that $\Dz F_1 = 0$ and $\Dz F_2 = 0$. The conclusion follows from Corollary \eqref{cor1}.
\end{proof}

\begin{theorem}\label{cor2}
Let $\varphi_1:\s^m\to\s^{n_1}(r_1)$ and $\varphi_2:\s^m\to\s^{n_2}(r_2)$ be two harmonic maps with constant energy densities, such that $r_1^2+r_2^2=1$. Then the map 
\begin{equation*}
\varphi = \i \circ \left(\varphi_1,\varphi_2\right):\s^m\to\s^{n_1+n_2+1},
\end{equation*}
 is harmonic if and only if  $\left|\d \varphi_1\right|^2 /r_1^2 = \left|\d \varphi_2\right|^2/r_2^2$. 
\end{theorem}
\begin{proof}
As $\f_1$ and $\f_2$ are harmonic with constant energy densities, from Proposition \eqref{pr4} there exist the unique non-negative integers $k_1$ and $ k_2$ such that 
\begin{align}\label{ec-13.1}
e(\f_1) &= \frac{1}{2} k_1r^2_1(k_1 + m - 1), \\
e(\f_2) &= \frac{1}{2} k_2r^2_2( k_2 + m - 1),\nonumber
\end{align} 
and there exist unique vector valued functions $ F_1:\r^{m+1}\to\r^{n_1+1}$ and $ F_2:\r^{m+1}\to\r^{n_2+1}$, such that each of their components is either a harmonic homogeneous polynomial of degree $k_1$, respectively $ k_2$, or the null polynomial, and they restrict to $\f_1$, respectively $\f_2$.
Applying Corollary \eqref{cor1} for $F_1$ and $F_2$, it follows that $\varphi$ is harmonic if and only if $k_1 = k_2$ and using Equation \eqref{ec-13.1} this is further equivalent to 
$$
\frac{1}{r_1^2}\left|\d \varphi_1\right|^2  = \frac{1}{r_2^2}\left|\d \varphi_2\right|^2.
$$
\end{proof}

\begin{theorem}\label{th7}
The bitension field of the map $\varphi$ given in Equation \eqref{ec-12} is given by
\begin{align}\label{ec-14}
  \tau_2(\varphi) =&\Dz\Dz F \nonumber \\
  &+ \left( 2(mk_1+k_1^2-3k_1-m+3)\Dz  F_1 + k_1^2(m+k_1-1)^2 \Phi_1,\right.\nonumber\\
  &\left. 2(mk_2+k_2^2-3k_2-m+3)\Dz  F_2 + k_2^2(m+k_2-1)^2 \Phi_2 \right)\nonumber  \\
   & + 2\left( \left|\dz F \right| ^2 -k_1^2r_1^2-k_2^2r_2^2\right)\nonumber\\
   & \cdot \left( -\Dz F_1 + k_1(1-m-k_1)\Phi_1,  -\Dz F_2 + k_2(1-m-k_2)\Phi_2\right)  \nonumber\\
   & +\left\{ -2\Dz \left( \left|\dz F \right| ^2 \right) - 2 \left|\stackrel{o}{\nabla} \dz F \right| ^2 + \left|\Dz  F \right|^2 \right. \nonumber \\ &\left.+2(m+2k_1-3)\left|\dz F_1 \right| ^2+2(m+2k_2-3)\left|\dz F_2 \right| ^2\right. \\
   & \left. -r_1^2k_1^2(m^2+4mk_1-6m+5k_1^2-7k_1+5) -2k_1(m+k_1-1)\left|\dz F_1 \right| ^2\right.\nonumber\\
   & \left. - r_2^2k_2^2(m^2+4mk_2-6m+5k_2^2-7k_2+5) - 2k_2(m+k_2-1)\left|\dz F_2 \right| ^2\right.\nonumber\\
   &\left.+2\left( \left|\dz F \right| ^2 - k_1^2r_1^2 - k_2^2r_2^2 \right)^2\right\}\Phi+2\dz F\left( \stackrel{o}{\textnormal{grad}}\left( \left|\dz F \right| ^2  \right)\right)\nonumber\\
   &-4\left( (k_1-1)\left|\dz F_1 \right| ^2 + (k_2-1) \left|\dz F_2 \right| ^2\right)(k_1\Phi_1,k_2\Phi_2).\nonumber
\end{align}
\end{theorem}

\begin{proof}
In the same setup as in the proof of Theorem \eqref{th6}, we continue to study the terms from the right-hand side of Equation \eqref{ec-1} in the case when the radius is $1$, aiming to express $\tau_2(\varphi)$ in terms of $F_1$ and $F_2$. 

Using Equation \eqref{ec-21}, on $\s^m$ we have
\begin{equation*}
  \textnormal{grad}\left(\left| \textnormal{d}\Phi\right|^2\right) = \stackrel{o}{\textnormal{grad}}\left( \left| \dz F\right|^2 \right) - \overline r \left( \left| \dz F\right|^2 \right).
\end{equation*}
Therefore, on $\s^m$
\begin{align}\label{ec-22}
  2\textnormal{d}\Phi\left( \textnormal{grad}\left(\left| \textnormal{d}\Phi\right|^2 \right)\right) & = 2\dz F\left(\stackrel{o}{ \textnormal{grad}}\left(\left| \dz F\right|^2\right) - \overline  r \left( \left| \dz F\right|^2  \right)\overline r\right)\nonumber\\
  & = 2\dz F\left(\stackrel{o}{ \textnormal{grad}}\left(\left| \dz F\right|^2\right)\right) - 2\left[  \overline r \left( \left| \dz F\right|^2 \right) \right]\cdot (k_1\Phi_1,k_2 \Phi_2).
\end{align}
We note that
\begin{align*}
  \left| \left( \dz F \right)(t\overline{x})\right|^2 &= \left| \left( \dz F_1 \right)(t\overline{x})\right|^2 + \left| \left( \dz F_2 \right)(t\overline{x})\right|^2 \\
   & = t^{2(k_1 - 1)}\left| \left( \dz F_1 \right)(\overline{x})\right|^2 + t^{2(k_2 - 1)} \left| \left( \dz F_2 \right)(\overline{x})\right|^2.
\end{align*}
It follows that
\begin{equation}\label{ec-23}
\overline r \left(\left|  \dz F \right|^2\right) = 2(k_1-1) \left|  \dz F_1 \right|^2 + 2(k_2-1)\left|  \dz F_2 \right|^2.
\end{equation}
From Equations \eqref{ec-22} and \eqref{ec-23}, on $\s^m$ we have
\begin{align}\label{ec-24}
  2\textnormal{d}\Phi\left( \textnormal{grad}\left(\left| \textnormal{d}\Phi\right|^2 \right)\right) =& 2\dz F\left(\stackrel{o}{ \textnormal{grad}}\left(\left| \dz F\right|^2\right)\right)\nonumber \\
  &- 2\left( 2(k_1-1)\left| \dz F_1\right|^2 + 2(k_2-1) \left| \dz F_2\right|^2\right) (k_1\Phi_1,k_2\Phi_2).
\end{align}
Further, we compute $\left|\tau(\Phi) \right|^2$, that is given by
\begin{equation}\label{ec-24.1}
\left|\tau(\Phi) \right|^2 = \left|\tau(\Phi_1) \right|^2 + \left|\tau(\Phi_2) \right|^2.
\end{equation}
From Equation \eqref{ec-19} it follows that
\begin{align*}
  \left|\tau(\Phi_1) \right|^2 &= \left|\tz(F_1) - k_1(m+k_1-1)\Phi_1 \right|^2 \\
   & = \left|\tz(F_1) \right|^2 + r_1^2k_1 ^2 (m+k_1-1)^2 - 2k_1(m+k_1-1) \left\langle \tz(F_1) ,\Phi_1 \right\rangle.
\end{align*}

As 
\begin{equation*}
\tau(\Phi_1) = \tau(\varphi_1) - \frac{1}{r_1^2}|\d \varphi_1|^2\Phi_1 = \tau(\varphi_1) - \frac{1}{r_1^2}|\d \Phi_1|^2\Phi_1.
\end{equation*}
using Equation \eqref{ec-19} we obtain
\begin{equation*}
   \tz\left(F_1\right) = \tau\left(\varphi_1 \right)- \frac{1}{r_1^2} \left| \textnormal{d} \Phi_1 \right|^2\cdot \Phi_1 + k_1(m+k_1-1)\Phi_1
\end{equation*}
Then, it follows that 
\begin{align}\label{ec-19.1}
  \left\langle \tz(F_1),\Phi_1 \right\rangle & = -\left| \textnormal{d} \Phi_1 \right|^2 + r_1^2k_1 (m+k_1-1)\\
   & = - \left|\dz F_1 \right|^2 + k_1 ^2 r_1^2 + r_1^2k_1 (m+k_1-1).\nonumber
\end{align}
Thus, on $\s^m$ we have
\begin{equation}\label{ec-24.2}
 \left|\tau(\Phi_1) \right|^2 =  \left|\tz(F_1)\right |^2 + r_1^2k_1^2 (m+k_1-1)^2 - 2k_1(m+k_1-1)\left(-\left| \dz F_1 \right|^2 + r_1^2k_1 (m+2k_1-1)\right).
\end{equation}
Similar we have for $\Phi_2$
\begin{equation}\label{ec-24.3}
 \left|\tau(\Phi_2) \right|^2 =  \left|\tz(F_2)\right |^2 + r_2^2k_2^2 (m+k_2-1)^2 - 2k_2(m+k_2-1)\left(-\left| \dz F_2 \right|^2 + r_2^2k_2 (m+2k_2-1)\right).
\end{equation}
Then, from Equations \eqref{ec-24.1}, \eqref{ec-24.2} and \eqref{ec-24.3} we obtain
\begin{align*}
  \left|\tau(\Phi) \right|^2 =&  \left|\tz(F_1)\right |^2 + \left|\tz(F_2)\right |^2  + 2k_1(m+k_1-1) \left| \dz F_1 \right|^2 + 2k_2(m+k_2-1) \left| \dz F_2 \right|^2\\
   & + r_1^2k_1^2 (m+k_1-1)^2 + r_2^2k_2^2 (m+k_2-1)^2\\
   & - 2r_1^2k_1^2 (m+k_1-1)(m+2k_1-1) - 2 r_2^2k_2^2 (m+k_2-1)(m+2k_2-1)\\
   =& \left|\tz(F_1)\right |^2 + \left|\tz(F_2)\right |^2  + 2k_1(m+k_1-1) \left| \dz F_1 \right|^2 + 2k_2(m+k_2-1) \left| \dz F_2 \right|^2\\
   &-r_1^2k_1^2\left(3k_1^2+4k_1(m-1)+(m-1)^2\right) -r_2^2k_2^2\left(3k_2^2+4k_2(m-1)+(m-1)^2\right).
\end{align*}
Next, using Equation \eqref{ec-20} it follows that
\begin{align*}
  \tau_2(\Phi) =& \ \Delta \Delta \Phi \\
               =& \left(\Delta \Delta \Phi_1, \Delta \Delta \Phi_2\right)\\
               =&  \left(\Dz\Dz F_1 + 2 \left(mk_1+k_1^2-3k_1 - m +3\right)\Dz F_1 + k_1^2 \left(m+k_1-1\right)^2\Phi_1,\right.\\
                &  \left.\Dz \Dz F_2 + 2 \left(mk_2+k_2^2-3k_2 - m +3\right)\Dz F_2 + k_2^2 \left(m+k_2-1\right)^2\Phi_2 \right).
\end{align*}
From Equation \eqref{ec-21} we obtain
\begin{align*}
  \Delta\left( \left| \textnormal{d}\Phi \right|^2 \right) =&\Delta \left( \left| \dz F \right|^2 \right) \\
   =&\Dz  \left( \left| \dz F \right|^2 \right) - (1-m)\overline r\left( \left| \dz F \right|^2 \right) + \overline r\left( \overline r\left(\left| \dz F \right|^2\right) \right)\\
   =&\Dz  \left( \left| \dz F \right|^2 \right) - (1-m)\left( 2(k_1-1)\left| \dz F_1 \right|^2 +  2(k_2-1)\left| \dz F_2 \right|^2\right)\\
   & + 4(k_1-1)^2 \left| \dz F_1 \right|^2 + 4(k_2-1)^2 \left| \dz F_2 \right|^2\\
   =&\Dz  \left( \left| \dz F \right|^2 \right) + 2(k_1-1)(m+2k_1-3)\left| \dz F_1 \right|^2 + 2(k_2-1)(m+2k_2-3)\left| \dz F_2 \right|^2
\end{align*}
In the following, first we will compute $\theta$, and then we will compute $\textnormal{div}\theta^\sharp$.
\begin{align*}
\theta\left (X\right) &= \left\langle\d\varphi\left (X\right),\tau(\varphi)\right\rangle = \left\langle\d\Phi\left (X\right),\tau(\Phi)\right\rangle\\
          &= \left\langle\dz F\left (X\right),\tz(F)-k(m+k-1)\Phi\right\rangle\\
          &= \left\langle\dz F\left (X\right),\tz(F)\right\rangle, \quad \textnormal{ on } \s^m.
\end{align*}
We know that 
\begin{equation*}
  \textnormal{div } \theta^{\sharp}=\sum_{k=1}^{m}\left\langle X_i, \nabla_{X_i}\theta^{\sharp}\right\rangle.
\end{equation*}
Then, at $\overline{x}$ we have
\begin{align*}
 \textnormal{div } \theta^{\sharp} & = \sum_{i=1}^{m}X_i\left\langle X_i, \theta^{\sharp}\right\rangle \\
   & = \sum_{i=1}^{m}X_i\left( \theta \left(X_i\right) \right) = \sum_{i=1}^{m}X_i\left\langle \dz F (X_i), \tz(F) \right\rangle   \\
   & = \sum_{i=1}^{m} \left\{ \left\langle\nz _{X_i} \dz F (X_i), \tz(F) \right\rangle + \left\langle \dz F (X_i),\nz _{X_i} \tz(F) \right\rangle  \right\}\\
   & =  \sum_{i=1}^{m} \left\{ \left\langle \left(\stackrel{o}{\nabla} \dz F\right) \left(X_i,X_i\right) + \dz F\left(\nz _{X_i},X_i  \right), \tz(F) \right\rangle + \left\langle \dz F (X_i),\nz _{X_i} \tz(F) \right\rangle  \right\} \\
  & = \sum_{i=1}^{m} \left\{ \left\langle \left(\stackrel{o}{\nabla} \dz F\right) \left(X_i,X_i\right) - rF, \tz(F) \right\rangle \right\}+ \left\langle \dz F, \dz  \left( \tz(F)\right) \right\rangle  - \left\langle r(F), r\left(\tz(F)\right)\right\rangle.
\end{align*}
Further,
\begin{align*}
  \textnormal{div } \theta^{\sharp}= & \left\langle\Dz  F + r\left(rF\right) - rF +m rF,\Dz  F\right\rangle + \left\langle \dz F, \dz  \left( \tz(F)\right) \right\rangle -  \left\langle r(F), r\left(\tz(F)\right)\right\rangle \\
   =&  \left\langle\Dz  F + (m-1)(k_1\Phi_1,k_2\Phi_2) + (k_1 ^2 \Phi_1, k_2^2 \Phi_2),\Dz  F\right\rangle\\
   & - \left\langle \dz F, \dz  \left(\Dz (F)\right) \right\rangle -  \left\langle \left(k_1  \Phi_1, \Phi_2 F_2\right), \left( (k_1-2)\Dz F_1, (k_2-2)\Dz F_2\right)\right\rangle \\
   =& \left|\Dz F \right|^2 + (m-1)k_1\left\langle \Phi_1,\Dz F_1\right\rangle+(m-1)k_2\left\langle \Phi_2,\Dz F_2\right\rangle\\
  & +k_1 ^2\left\langle \Phi_1,\Dz F_1\right\rangle + k_2 ^2 \left\langle \Phi_2,\Dz F_2\right\rangle - \left\langle \dz F, \dz  \left(\Dz (F)\right) \right\rangle \\
  & +k_1 (k_1-2)\left\langle \Phi_1,\Dz F_1\right\rangle + k_2 (k_2-2) \left\langle \Phi_2,\Dz F_2\right\rangle \\
  =& \left|\Dz F \right|^2 - \left\langle \dz F, \dz  \left(\Dz (F)\right) \right\rangle +k_1 (m+2k_1-3)\left\langle \Phi_1,\Dz F_1\right\rangle\\
   &+ k_2 (m+2k_2-3)\left\langle \Phi_2,\Dz F_2\right\rangle.
\end{align*}
Then, using Equation \eqref{ec-19.1} we obtain
\begin{align*}
  \textnormal{div } \theta^{\sharp} =& \left|\Dz F \right|^2 - \left\langle \dz F, \dz  \left(\Dz (F)\right) \right\rangle\\
  & + k_1 (m+2k_1-3)\left( \left|\dz F_1 \right| ^2 - r_1^2k_1^2 - r_1^2 k_1(m+k_1-1) \right) \\
   & + k_2 (m+2k_2-3)\left( \left|\dz F_2 \right| ^2 - r_2^2k_2^2 - r_2^2 k_2(m+k_2-1) \right).
\end{align*}

We know that $\dz\Dz  =\Dz  \dz  $. It follows that
\begin{equation*}
\left\langle \dz F, \dz  \left(\Dz F\right) \right\rangle = \left\langle \dz F,\Dz  \left( \dz F\right) \right\rangle.
\end{equation*}
From the Weitzenb$\stackrel{..}{\textnormal{o}}$ck formula for $\dz F$ we get 
\begin{equation*}
  \frac{1}{2}\Dz  \left( \left| \dz F \right|^2 \right) = \left\langle \dz F,\Dz \dz F \right\rangle + \left|\nz \dz F \right|^2.
\end{equation*}
It follows that
\begin{equation*}
\left \langle \dz F,\Dz  \left( \dz F\right) \right \rangle =  \frac{1}{2}\Dz  \left( \left| \dz F \right|^2 \right) + \left|\nz \dz F \right|^2.
\end{equation*}
So, using the above result we have
\begin{align*}
 \textnormal{div } \theta^{\sharp}= & \left|\Dz F \right|^2 -  \frac{1}{2}\Dz  \left( \left| \dz F \right|^2 \right) - \left|\nz \dz F \right|^2 \\
  & + k_1 (m+2k_1-3)\left( \left|\dz F_1 \right| ^2  - r_1^2 k_1(m+2k_1-1) \right) \\
   & + k_2 (m+2k_2-3)\left( \left|\dz F_2 \right| ^2  - r_2^2 k_2(m+2k_2-1) \right).
\end{align*}
We replace in Equation \eqref{ec-1} the terms we have computed and we obtain
\begin{align*}
  \tau_2(\varphi) =&\left(\Dz\Dz F_1 + 2(mk_1+k_1^2-3k_1-m+3)\Dz  F_1 + k_1^2(m+k_1-1)^2 \Phi_1,\right.\\
  &\left.\Dz\Dz F_2 + 2(mk_2+k_2^2-3k_2-m+3)\Dz  F_2 + k_2^2(m+k_2-1)^2 \Phi_2 \right)  \\
   & + 2\left( \left|\dz F_1 \right| ^2 + \left|\dz F_2 \right| ^2 -k_1^2r_1^2-k_2^2r_2^2\right)\\
   & \cdot \left( -\Dz F_1 + k_1(1-m-k_1)\Phi_1,  -\Dz F_2 + k_2(1-m-k_2)\Phi_2\right)  \\
   & +\left\{ -2\Dz \left( \left|\dz F \right| ^2 \right) - 2 \left|\stackrel{o}{\nabla} \dz F \right| ^2 + \left|\Dz  F \right|^2 +2(m+2k_1-3)\left|\dz F_1 \right| ^2\right. \\
   & \left. +2(m+2k_2-3)\left|\dz F_2 \right| ^2 -r_1^2k_1^2(m^2+4mk_1-6m+5k_1^2-7k_1+5) \right.\\
   & \left. - r_2^2k_2^2(m^2+4mk_2-6m+5k_2^2-7k_2+5) -2k_1(m+k_1-1)\left|\dz F_1 \right| ^2\right.\\
   &\left. - 2k_2(m+k_2-1)\left|\dz F_2 \right| ^2+2\left( \left|\dz F_1 \right| ^2 + \left|\dz F_2 \right| ^2 - k_1^2r_1^2 - k_2^2r_2^2 \right)^2\right\}\Phi\\
   &-4\left( (k_1-1)\left|\dz F_1 \right| ^2 + (k_2-1) \left|\dz F_2 \right| ^2\right)(k_1\Phi_1,k_2\Phi_2)\\
   & +2\dz F\left( \stackrel{o}{\textnormal{grad}}\left( \left|\dz F \right| ^2  \right)\right).
\end{align*}
\end{proof}

\begin{remark}
In the case $k_1 = k_2 = k$, we obtain a particular case of Theorem \eqref{th4} for the case when the radius is $r=1$.
\end{remark}

\begin{example}To test the Theorem \eqref{th7}, we now construct an explicit map from $\s^1$ to $\s^3$. Although this map meets the conditions of the theorem, we verify its behavior through a direct computation. 

As we already discussed in Example \eqref{ex1},  an orthogonal basis, with respect to the usual product, for the linear space of homogeneous harmonic polynomials of degree $k$ in 2 variables is given in polar coordinates by $r^k\cos(k\theta)$ and $r^k\sin(k\theta)$. 

Let $z=x+iy$ and 
\begin{equation*}
P_k(x,y)=\textnormal{Re}(z^k) \quad \textnormal{ and } \quad   Q_k(x,y)=\textnormal{Im}(z^k).
\end{equation*}
Then $P_k(x,y)$ and $Q_k(x,y)$ form a basis for the linear space of homogeneous harmonic polynomials of degree $k$ in 2 variables. It follows that for the $k$-form $G_k:\r^2\to\r^2$ given by $G_k=\left(P_k, Q_k\right)$ we have on $\r^2$
\begin{align}\label{ec-30}
\left|\dz G_k \right|^2 &= 2\cdot k^2 \left| z \right|^{2(k-1)} \\
                        &= 2\cdot k^2 (x^2+y^2)^{k-1}\nonumber
\end{align}

Now, let $k_1$ and $k_2$ be non-negative integers, $k_1\neq k_2$. We consider the degrees $k_1$ and $k_2$ and we define the vector functions $F_{1}, \ F_{2} :\r^2\to\r^2$ by $F_{1} = r_1G_{k_1}$ and $F_{2} = r_2G_{k_2}$, such that $r_1^2+r_2^2=1$. Now we define the vector function $F:\r^2\to\r^4$ defined by $F=\left( F_{1}, F_{2}\right)$. We want to see when the map $\varphi$ defined as in Equation \eqref{ec-12}, i.e. the restriction of $F$, is proper biharmonic.

It is clear that 
\begin{equation}\label{ec-31}
\Dz F_{1} = \Dz F_{2}  = 0.
\end{equation}

From Equation \eqref{ec-30} we have on $\r^{m+1}$
\begin{align}\label{ec-32}
\left|\dz F_{1} \right|^2 &= 2 r_1^2 k_1^2 (x^2+y^2)^{k_1-1}. \\
\left|\dz F_{2} \right|^2 &= 2 r_2^2 k_2^2 (x^2+y^2)^{k_2-1}.\nonumber
\end{align}

As $F_1$ and $F_2$ are harmonic, from Proposition \eqref{pr6} we obtain on $\s^m$
\begin{align}\label{ec-33}
\Dz\left(\left|\dz F_{1}\right|^2\right) =&- 8r_1^2 k_1^2\left(k_1-1\right), \\
\Dz\left(\left|\dz F_{2}\right|^2\right) =&- 8r_2^2 k_2^2\left(k_2-1\right),\nonumber
\end{align}
and 
\begin{align}\label{ec-34}
\left| \nz \dz F_{1}\right|^2 =& 8r_1^2 k_1^2\left(2k_1-1\right), \\
\left| \nz \dz F_{2}\right|^2 =& 8r_2^2 k_2^2\left(2k_2-1\right).\nonumber
\end{align}

Then by direct computations, on $\s^{m}$ we have
\begin{align}\label{ec-35}
\dz F_1\left(\stackrel{o}{\textnormal{grad}}\left(\left|\dz F_{1}\right|^2\right)\right) =& 4r_1^2 k_1^2\left(k_1-1\right)F_{k_1}, \\
\dz F_2\left(\stackrel{o}{\textnormal{grad}}\left(\left|\dz F_{2}\right|^2\right)\right) =& 4r_2^2 k_2^2\left(k_2-1\right)F_{k_2}.\nonumber
\end{align}

Replacing Equations \eqref{ec-31}, ..., \eqref{ec-35} in Equation \eqref{ec-14} we obtain 
\begin{align*}
\tau_2(\varphi) =& (2r_1^2-1)(k_1-k_2)^2(k_1+k_1)^2\left((r_1^2-1)F_{1}, r_1^2 F_{2}\right). 
\end{align*}
As $k_1\neq k_2$, from Corollary \eqref{cor1} it follows that the map $\f$ is not harmonic. Therefore, the map  $\varphi$ is proper biharmonic if and only if $r_1 = r_2= 1/\sqrt{2}$.

We note that, as in Example \eqref{ex1}, the proper biharmonic map $\varphi$ constructed is part of the same larger family of proper biharmonic curves in $\s^3$ from \cite{CMO01} .
\end{example}

\begin{example} Now, we test Theorem \eqref{th7} for the case when $F_1$ and $F_2$ are not harmonic.

Let $F_1:\r^4\to\r^4$ and $F_2:\r^4\to\r$ given by
\begin{align*}
F_1\left(\overline x\right) =& \left(\frac{1}{\sqrt{2}}\left(\left(x^1\right)^2 + \left(x^2\right)^2 - \left(x^3\right)^2 - \left(x^4\right)^2\right), \sqrt{2}\left(x^1x^3 - x^2x^4\right),\right.\\
                           & \left. \frac{1}{\sqrt{2}}\left(\sqrt{2}\left(x^1x^4 + x^2x^3\right)-\frac{1}{2}\left|\overline x\right|^2\right),\frac{1}{\sqrt{2}}\left(\sqrt{2}\left(x^1x^4 + x^2x^3\right)+\frac{1}{2}\left|\overline x\right|^2\right)\right).
\end{align*}
and
$$
F_2(\overline x) = \frac{1}{2}\left|\overline x\right|^4.
$$
It is easy to see that $\left|F_1(\overline x)\right|^2 = (3/4)\left|\overline x\right|^4$. 

Now, consider the map $\varphi$ as in Equation \eqref{ec-12}. 
By direct calculations, on $\r^4$ we have
$$
\left|\dz F_1\right|^2 = 7\left|\overline x\right|^2 \quad \textnormal{ and } \quad \left|\dz F_2\right|^2 = 4\left|\overline x\right|^6,
$$
and on $\s^3$ we have
\begin{align*}
  \Dz F_1 =& \left(0,0,\frac{4}{\sqrt{2}}, -\frac{4}{\sqrt{2}}\right), & \Dz \Dz F_1&=0, \\
  \Dz F_2 =& -12, & \Dz\Dz F_2 &= 96.
\end{align*}

Since $F_1$ is a quadratic form, using Proposition \eqref{pr2} it follows that
$$
-2\Dz \left( \left|\dz F_1 \right| ^2 \right) - 2 \left|\stackrel{o}{\nabla} \dz F_1 \right| ^2 + \left|\Dz  F_1 \right|^2 = 72.
$$
For $F_2$ by direct calculations we obtain
$$
-2\Dz \left( \left|\dz F_2 \right| ^2 \right) - 2 \left|\stackrel{o}{\nabla} \dz F_2 \right| ^2 + \left|\Dz  F_2 \right|^2 = 432.
$$
Then, on $\s^3$ we have
$$
\dz F\left(\stackrel{o}{\textnormal{grad}}\left(\left|\dz F\right|^2\right)\right) = 76\left(F_1, 1\right).
$$
We replace all in Equation \eqref{ec-14} and we obtain $\tau_2(\varphi) = 0$ on $\s^3$, thus $\varphi$ is biharmonic. 
Using Theorem \eqref{th6}, by direct calculations we obtain 
$$
\tau(\varphi) = \left(0,0,-\frac{4}{\sqrt{2}},\frac{4}{\sqrt{2}},12\right) - \left(4F_1, 10\right).
$$
Thus, $\varphi$ is not harmonic, therefore it is proper-biharmonic.
\end{example}
\begin{remark}
We can rewrite $\varphi$ as 
\begin{align*}
\varphi\left(\overline x\right) =& \left(\frac{1}{\sqrt{2}}\left(\left(x^1\right)^2 + \left(x^2\right)^2 - \left(x^3\right)^2 - \left(x^4\right)^2\right), \sqrt{2}\left(x^1x^3 - x^2x^4\right),\right.\\
                           & \left. \frac{1}{\sqrt{2}}\left(\sqrt{2}\left(x^1x^4 + x^2x^3\right)-\frac{1}{2}\left|\overline x\right|^2\right),\frac{1}{\sqrt{2}}\left(\sqrt{2}\left(x^1x^4 + x^2x^3\right)+\frac{1}{2}\left|\overline x\right|^2\right), \frac{1}{2}\left|\overline x\right|^2\right).
\end{align*} 
We note that this example first appeared in \cite{A23}. We note that although the last component is constant on the sphere, the first $4$ components do not form an harmonic map.

However, by applying orthogonal transformations on the components of $\varphi$ we can bring the map to
\begin{align*}
\varphi\left(\overline x\right) =& \left(\frac{1}{\sqrt{2}}\left(\left(x^1\right)^2 + \left(x^2\right)^2 - \left(x^3\right)^2 - \left(x^4\right)^2\right), \sqrt{2}\left(x^1x^3 - x^2x^4\right),\right.\\
                           & \left. \sqrt{2}\left(x^1x^4 + x^2x^3\right), \frac{1}{\sqrt{2}}\left(\left(x^1\right)^2 + \left(x^2\right)^2 + \left(x^3\right)^2 + \left(x^4\right)^2\right), 0\right).
\end{align*}
The first $3$ components of $\varphi$ form an harmonic map $\psi:\s^3\to \s^2\left(1/\sqrt{2}\right)$ and $\varphi$ is of the construction type outlined in Theorem \eqref{th1}.
\end{remark}

\begin{example}
Let $F_1:\r^3\to\r^5$ a quadratic form such that its restriction $\varphi_1$ is the Veronese map 
\begin{equation*}
  F_1(x,y,z)=r_1\left(\frac{1}{2}\left( x^2+y^2-2z^2 \right), \frac{\sqrt{3}}{2}\left(x^2-y^2 \right),\sqrt{3}xy,\sqrt{3}xz, \sqrt{3}yz\right).
\end{equation*}
and consider $F_2:\r^3\to\r^7$ a form of degree $3$ (that first appeared in \cite{CCK71}), given by  
\begin{align*}
  F_2(x,y,z)=&r_2\left( \frac{1}{2}z\left( -3x^2-3y^2+2z^2\right),\frac{\sqrt{6}}{4}x\left(-x^2-y^2+4z^2 \right), \frac{\sqrt{15}}{2}z\left(x^2-y^2 \right), \right.\\ &\left.\frac{\sqrt{10}}{4}x\left(x^2-3y^2\right),\frac{\sqrt{6}}{4}y\left(-x^2-y^2+4z^2 \right), \sqrt{15}xyz, \frac{\sqrt{10}}{4} y\left(3x^2-y^2\right)\right),
\end{align*}
such that $r_1^2+r_2^2=1$.

Consider the map $\varphi$ as in Equation \eqref{ec-12}. 

In this case, we have on $\r^3$
\begin{equation*}
\Dz F_1=0, \quad \Dz F_2=0
\end{equation*}
and on $\r^3$
\begin{equation*}
\left|\dz F_1\right|^2 = 10 r_1^2|\overline x|^2, \quad \left|\dz F_2\right|^2 = 21 r_2^2|\overline x|^4.
\end{equation*}
As $F_1$ is a quadratic form, using Proposition \eqref{pr2} it follows that on $\s^2$ we have
$$
-2\Dz \left( \left|\dz F_1 \right| ^2 \right) - 2 \left|\stackrel{o}{\nabla} \dz F_1 \right| ^2 + \left|\Dz  F_1 \right|^2 = 60r_1^2.
$$
For $F_2$, using Proposition \eqref{pr6} we obtain on $\s^2$
$$
-2\Dz \left( \left|\dz F_2 \right| ^2 \right) - 2 \left|\stackrel{o}{\nabla} \dz F_2 \right| ^2 + \left|\Dz  F_2 \right|^2 = 420r_2^2.
$$
Further, on $\s^m$ we have
$$
\dz F\left(\stackrel{o}{\textnormal{grad}}\left(\left|\dz F\right|^2\right)\right) = 2(10r_1^2+42r_2^2)\left(2F_1,3F_2\right).
$$
We replace in Equation \eqref{ec-14} and we obtain 
$$
\tau_2(\f) = 12(1-2r_1^2)\left(F_1, -2F_2\right).
$$
Therefore, $\tau_2(\varphi) = 0$ if and only if $r_1 = r_2 = 1/\sqrt{2}$.
\end{example}

\begin{theorem}\label{th8}
If $F_1:\r^{m+1}\to\r^{n_1+1}$ and $F_2:\r^{m+1}\to\r^{n_2+1}$ are two harmonic forms of degree $k_1$, respectively $k_2$, then the map $\varphi$ given in Equation \eqref{ec-12} is proper biharmonic if and only if $r_1 = r_2 = 1/\sqrt{2}$ and $k_1\neq k_2$. 
\end{theorem}
\begin{proof}
Using Theorem \eqref{th3} the conditions $\Dz F_1 = 0$ and $\Dz F_2 = 0$ imply that on $\r^{m+1}$ we have
\begin{align}\label{ec-40}
\left|\dz F_1\right|^2 & = k_1r_1^2(m+2k_1-1)\left|\overline x\right|^{2(k_1-1)}, \\
\left|\dz F_2\right|^2 & = k_2r_2^2(m+2k_2-1)\left|\overline x\right|^{2(k_1-1)}. \nonumber
\end{align}
Now, using Proposition \eqref{pr6} we obtain on $\s^m$ 
\begin{align}\label{ec-41}
\Dz\left(\left|\dz F_1\right|^2\right) & =- 2r_1^2k_1(k_1-1)(m+2k_1-1)(m+2k_1-3), \\
\Dz\left(\left|\dz F_2\right|^2\right) & =- 2r_2^2k_2(k_2-1)(m+2k_2-1)(m+2k_2-3).\nonumber
\end{align}
and 
\begin{align}\label{ec-42}
\left|\nz \dz F_1\right|^2 &=r_1^2k_1(k_1-1)\left(m^2-4m+3+4k_1(m-2)+4k_1^2\right),\\
\left|\nz \dz F_2\right|^2 &=r_2^2k_2(k_2-1)\left(m^2-4m+3+4k_2(m-2)+4k_2^2\right).\nonumber
\end{align} 
Next, using Equation \eqref{ec-40} we obtain on $\s^m$ by direct computations
\begin{align}\label{ec-43}
2\dz F\left(\stackrel{o}{ \textnormal{grad}}\left(\left| \dz F\right|^2\right)\right)=  &- 2\left( 2(k_1-1)\left| \dz F_1\right|^2 + 2(k_2-1) \left| \dz F_2\right|^2\right) (k_1\Phi_1,k_2\Phi_2).
\end{align}

Taking into account that $r_1^2+r_2^2=1$ and replacing Equations \eqref{ec-40}, ..., \eqref{ec-43} in Equation \eqref{ec-14} we obtain by direct calculations that the bitension field of the map $\varphi$ is given by
\begin{align}\label{ec-44}
\tau_2(\varphi) =& (2r_1^2-1)(k_1-k_2)^2(m+k_1+k_1-1)^2\left((r_1^2-1)\Phi_1, r_1^2 \Phi_2\right). 
\end{align}
We recall from Corollary \eqref{cor1}  that  $k_1 = k_2$ is equivalent in this case to $\varphi$ being harmonic. Thus, we impose $k_1\neq k_2$.
From Equation \eqref{ec-44} it follows directly that the map $\varphi$ is proper biharmonic if and only if $r_1 = 1/\sqrt{2}$.
\end{proof}

\begin{remark}
The above theorem can be viewed as a particular case of a more general result recently established in \cite{B25} (see Theorem 1.3). However, in \cite{B25}, the author employs a different method in a broader setting, while in this paper Theorem \eqref{th8} follows as an application to Theorem \eqref{th7}.
\end{remark}

\begin{corollary}\label{cor5}
If $F_1:\r^{m+1}\to\r^{n_1+1}$ and $F_2:\r^{m+1}\to\r^{n_2+1}$ are forms of degree $k_1$, respectively $k_2$, each of them of minimal degree in its class, such that their restrictions $\varphi_1:\s^m\to\s^{n_1}$ and $\varphi_2:\s^m\to\s^{n_2}$ are harmonic, then the map $\f$ given in Equation \eqref{ec-12} is proper biharmonic if and only if $k_1 \neq k_2$ and $r_1 =r_2 =1/\sqrt{2}$. 
\end{corollary}
\begin{proof}
Using Proposition \eqref{pr5} it follows that $\Dz F_1 = 0$ and $\Dz F_2 = 0$. The conclusion follows from Theorem \eqref{th8}.
\end{proof}

\begin{theorem}
Let $\varphi_1:\s^m\to\s^{n_1}(r_1)$ and $\varphi_2:\s^m\to\s^{n_2}(r_2)$ be two harmonic maps with constant energy densities, such that $r_1^2+r_2^2=1$. Then the map 
\begin{equation*}
\varphi = \i \circ \left(\varphi_1,\varphi_2\right):\s^m\to\s^{n_1+n_2+1},
\end{equation*}
is proper biharmonic if and only if  $r_1 =r_2 =1/\sqrt{2}$ and $e(\f_1)\neq e(\f_2)$.
\end{theorem}
\begin{proof}
As in Corollary \eqref{cor2}, since $\f_1$ and $\f_2$ are harmonic with constant energy densities, from Proposition \eqref{pr4} it follows that there exist the unique non-negative integers $k_1$ and $k_2$ such that 
\begin{align*}
e(\f_1) &= \frac{1}{2}k_1r_1^2( k_1 + m - 1), \\
e(\f_2) &= \frac{1}{2}k_2r_2^2( k_2 + m - 1),
\end{align*} 
and there exist unique vector valued functions $ F_1:\r^{m+1}\to\r^{n_1+1}$ and $ F_2:\r^{m+1}\to\r^{n_2+1}$, such that each of their components is either a harmonic homogeneous polynomial of degree $k_1$, respectively $ k_2$, or the null polynomial, and they restrict to $\f_1$, respectively $\f_2$. 
Applying Theorem \eqref{th8} it follows that $\varphi$ is proper biharmonic if and only if $k_1 \neq k_2$ and $r_1 =r_2 =1/\sqrt{2}$. In this case, as the condition $k_1 \neq k_2$ is equivalent to $e(\f_1)\neq e(\f_2)$, the conclusion follows.
\end{proof}

As an application to Theorem \eqref{th8} we construct a new family of proper biharmonic maps starting from standard minimal immersions (one can see \cite{BW03}, \cite{dCW71} or \cite{ER93}).

Let $\left\{\Phi_1,\ldots,\Phi_{n(k)}\right\}$ be a basis for the vector space of spherical harmonics of order $k$ which is orthonormal with respect to the usual product, 
 where $n(k) = (m+2k-1)\left((m+k-2)!\right)/\left(k!(m-1)\right)$. Consider the map 
$$
\Phi = c(k)\left(\Phi_1,\ldots,\Phi_{n(k)}\right):\s^m\to\r^{n(k)},
$$
where $c(k)$ is a positive constant to be chosen. The image of $\Phi$ lies in a sphere $\s^{n(k)-1}\left (r\right)$, where the radius $r = \sqrt{m/\left(k(m+k-1)\right)}$, and  we can write $\Phi = \i\circ\varphi$, where $\f:\s^m\to\s^{n(k)-1}\left (r\right)$ is harmonic. As homothetic changes of the domain or codomain metrics preserve the harmoncity and biharmonicity, we can assume that $\varphi$ maps $\s^m$ into $\s^{n(k)-1}\left(1/\sqrt{2}\right)$.
\begin{theorem}
Let $k_1\neq k_2$ be two non-negative integers and let 
$$
\varphi_1:\s^m\to\s^{n(k_1)-1}\left(\frac{1}{\sqrt{2}}\right)\quad \textnormal{ and }\quad \varphi_2:\s^m\to\s^{n(k_2)-1}\left(\frac{1}{\sqrt{2}}\right)
$$
 be two harmonic maps constructed as above. Then the map 
$$
\varphi  = i\circ\left( \varphi_1,\varphi_2\right):\s^m\to\s^{n(k_1)+n(k_2)-1}
$$
is proper biharmonic.
\end{theorem}

Encouraged by the positive results of our initial method, we adopt a different approach, seeking to extend our previous findings in a broader context. Now we consider the special case when 
\begin{equation}\label{ec-60}
\varphi = \i \circ \left(\varphi_1,\varphi_2\right):\s^{m_1}\times\s^{m_2}\to\s^n,
\end{equation}
where $n = n_1+n_2+1$, $\i$ is the canonical inclusion of the standard product $\s^{n_1}(r_1)\times\s^{n_2}(r_2)$ in $\s^n$,  and $\varphi_1$ and $\varphi_2$ are given as in the below diagrams

\bigskip
\begin{center}
  \begin{tikzpicture}
  \matrix (m) [matrix of math nodes,row sep=3em,column sep=4em,minimum width=2em]
  {
     \r^{m_1+1} & \r^{n_1+1} \\
     \s^{m_1} & \s^{n_1}(r_1) \\};
  \path[-stealth]
    (m-2-1) edge node [left] {$\i$} (m-1-1)
    (m-1-1) edge node [above] {$F_1$} (m-1-2)
    (m-2-1) edge node [above] {$\Phi_1$} (m-1-2)
    (m-2-2) edge node [right] {$\i_1$} (m-1-2)
    (m-2-1) edge node [below] {$\varphi_1$} (m-2-2);
\end{tikzpicture} 
\begin{tikzpicture}
  \matrix (m) [matrix of math nodes,row sep=3em,column sep=4em,minimum width=2em]
  {
     \r^{m_2+1} & \r^{n_2+1} \\
     \s^{m_2} & \s^{n_2}(r_2) \\};
  \path[-stealth]
    (m-2-1) edge node [left] {$\i$} (m-1-1)
    (m-1-1) edge node [above] {$F_2$} (m-1-2)
    (m-2-1) edge node [above] {$\Phi_2$} (m-1-2)
    (m-2-2) edge node [right] {$\i_2$} (m-1-2)
    (m-2-1) edge node [below] {$\varphi_2$} (m-2-2);
\end{tikzpicture}
\end{center}
where $r_1^2 + r_2^2 = 1$ and $F_1$ and $F_2$ are forms of degree $k_1$, respectively $k_2$, i.e. on $\r^{m_1+1}$, respectively $\r^{m_2}$, we have $\left|F_1\left(\overline{x}\right)\right|^2 = r_1^2\left|\overline x\right|^{2k_1}$, respectively $\left|F_2\left(\overline{x}\right)\right|^2 = r_2^2\left|\overline x\right|^{2k_2}$.
Then, for $\varphi$ we have the following diagram
\bigskip
\begin{center}
  \begin{tikzpicture}
  \matrix (m) [matrix of math nodes,row sep=6em,column sep=7em,minimum width=4em]
  {
     \r^{m_1+m_2+1} & \r^{n_1+n_2+2} \\
     \s^{m_1}\times\s^{m_2} & \s^{n_1+n_2+1} \\};
  \path[-stealth]
    (m-2-1) edge node [left] {$\i$} (m-1-1)
    (m-1-1) edge node [above] {$F=(F_1,F_2)$} (m-1-2)
    (m-2-1) edge node [above] {$\Phi=(\Phi_1,\Phi_2)$} (m-1-2)
    (m-2-2) edge node [right] {$\i$} (m-1-2)
    (m-2-1) edge node [below] {$\varphi=\i\circ(\varphi_1,\varphi_2)$} (m-2-2);
\end{tikzpicture} 
\end{center}

\begin{theorem}\label{th10}
Consider the map $\varphi$ given in Equation \eqref{ec-60}, such that $F_1$ and $F_2$ are harmonic. Then $\varphi$ is proper biharmonic if and only if $r_1 =r_2 = 1/\sqrt{2}$ and $e\left( \varphi_1 \right) \neq e\left( \varphi_2 \right)$.
\end{theorem}
\begin{proof}
Consider $\{E_i\}^{m_1}_{i=1}$  and $\{F_j\}^{m_2}_{j=1}$, two geodesic frame fields around $\overline{x}_0$ in $\s^{m_1}$, respectively $\overline{y}_0\in\s^{m_2}$. Then $\{(E_i,0)\}^{m_1}_{i=1}\cup \{(0,F_j)\}$ is a geodesic frame field around $(\overline x_0,\overline y_0)\in U\times V \subset\s^{m_1}\times\s^{m_2}$.

We will compute the bitension field of the map $\varphi$ given in Equation \eqref{ec-1} in terms of $F_1$ and $F_2$.

Using the same type of reasoning as in Equations \eqref{ec-15} and \eqref{ec-17} for the individual cases of $\Phi_1$ and $\Phi_2$, on $U\times V$ we have
\begin{align}\label{ec-61}
  \left|\textnormal{d} \Phi \right|^2_{(\overline{x},\overline{y})} =& \left|\textnormal{d} \Phi_1 \right|^2_{\overline{x}} + \left|\textnormal{d} \Phi_2 \right|^2_{\overline{y}} \nonumber\\
   =& \left|\dz F_1 \right|^2_{\overline{x}} - k_1^2r_1^2 + \left|\dz F_2 \right|^2_{\overline{y}} - k_2^2r_2^2\\
   =& k_1 r_1^2 \left( m_1 + k_1 -1 \right) + k_2 r_2^2 \left( m_2 + k_2 -1 \right),\nonumber
\end{align}
which is constant on $\s^{m_1}\times\s^{m_2}$.
It follows that 
\begin{equation}\label{ec-62}
\Delta\left|\textnormal{d}\Phi\right|^2=0
\end{equation}
and 
\begin{equation}\label{ec-63}
\textnormal{d}\Phi\left(\textnormal{grad}\left(\left|\textnormal{d}\Phi\right|^2\right)\right)=0.
\end{equation}
Next, similar to how we obtained Equation \eqref{ec-19}, in our case we get
\begin{align}\label{ec-64}
\tau(\Phi) =& \left( \tau\left(\Phi_1 \right),\tau\left(\Phi_2 \right) \right)\nonumber  \\
    =& \left(-\Dz F_1 - k_1(m_1+k_1-1)\Phi_1,-\Dz F_2 - k_2(m_2+k_2-1)\Phi_2 \right) \\
    =& \left(- k_1(m_1+k_1-1)\Phi_1, - k_2(m_2+k_2-1)\Phi_2 \right),\nonumber
\end{align}
that is equivalent to
\begin{align}\label{ec-65}
  \Delta \Phi =& \left( k_1(m+k_1-1)\Phi_1, k_2 (m+k_2-1)\Phi_2 \right), \quad \textnormal{ on } \s^{m_1}\times\s^{m_2}.
\end{align} 
Thus,
\begin{align}\label{ec-66}
 \left|\tau(\Phi) \right|^2 = k_1^2 r_1^2 (m_1+k_1-1)^2 + k_2^2r_2^2(m_2+k_2-1)^2.
\end{align}
Now we compute $\theta$.
\begin{align*}
\theta(X,Y) =& \left\langle \textnormal{d}\varphi(X,Y), \tau(\varphi) \right\rangle \\
   =& \left\langle \textnormal{d}\Phi(X,Y), \tau(\Phi) \right\rangle \\
   =& \left\langle \textnormal{d}\Phi(X,Y), \left( \tau(\Phi_1), \tau(\Phi_2)\right)\right\rangle.
\end{align*}
Therefore, 
\begin{align}\label{ec-67}
  \theta(X,Y) =& \left\langle \left( \textnormal{d}\Phi_1\left (X\right), \textnormal{d}\Phi_2(Y)\right), \left( - k_1(m_1+k_1-1)\Phi_1, - k_2(m_2+k_2-1)\Phi_2 \right) \right\rangle \nonumber \\
  =&  - k_1(m_1+k_1-1) \left\langle \textnormal{d}\varphi_1\left (X\right), \Phi_1 \right\rangle - k_2(m_2+k_2-1)\left\langle \textnormal{d}\varphi_2(Y), \Phi_2  \right\rangle \\
  =& 0.\nonumber
\end{align}
Using Equation \eqref{ec-65} we have
\begin{align}\label{ec-68}
\tau_2(\Phi) =& \Delta\Delta \Phi\nonumber \\
  =& \left( \Delta\Delta \Phi_1, \Delta\Delta \Phi_2 \right) \\
  =& \left( \Delta\left(k_1(m_1+k_1-1)\Phi_1\right), \Delta\left(k_2(m_2+k_2-1)\Phi_2\right) \right) \nonumber \\
  =& \left(  k_1^2(m_1+k_1-1)^2\Phi_1,  k_2^2(m_2+k_2-1)^2\Phi_2  \right).\nonumber
\end{align}
Now, replacing Equations \eqref{ec-61}, ..., \eqref{ec-68} in Equation \eqref{ec-1} we obtain
\begin{align}\label{ec-69}
\tau_2(\varphi)  =& \left(  k_1^2(m_1+k_1-1)^2\Phi_1,  k_2^2(m_2+k_2-1)^2\Phi_2  \right)\nonumber\\ 
  & + \left(-k_1^2 r_1^2 (m_1+k_1-1)^2 - k_2^2r_2^2(m_2+k_2-1)^2 \right. \nonumber\\
  &\left.+ 2\left(k_1 r_1^2 (m_1+k_1-1) + k_2r_2^2(m_2+k_2-1) \right)^2 \right) \left(\Phi_1, \Phi_2 \right)\nonumber\\
  &+2\left(k_1 r_1^2 (m_1+k_1-1) + k_2 r_2^2(m_2+k_2-1) \right)\\
  &\cdot \left( - k_1(m_1+k_1-1)\Phi_1,  -k_2(m_2+k_2-1)\Phi_2  \right).\nonumber\\
  =&(2r_1^2-1)\left( \frac{1}{r_1^2}\left|\textnormal{d}\Phi_1 \right|^2 - \frac{1}{r_2^2}\left|\textnormal{d}\Phi_2 \right|^2\right)\left( (r_1^2-1)\Phi_1,r_1^2\Phi_2 \right).\nonumber
\end{align}
Now we study the harmonicity of $\varphi$. Using Equation \eqref{ec-61} we obtain
\begin{align}\label{ec-70}
  \tau(\varphi) =& \tau(\Phi) + \left|\textnormal{d}\Phi \right|^2 \Phi\nonumber \\
  =& \left(-k_1(m_1+k_1-1)\Phi_1, -k_2(m_2+k_2-1)\Phi_2 \right)  \nonumber\\
  &+ \left(  k_1r_1^2(m_1+k_1-1) + k_2r_2^2(m_2+k_2-1)^2 \right)\left(\Phi_1, \Phi_2 \right)\nonumber\\
  =& (r_1^2-1)\left( k_1(m_1+k_1-1) - k_2(m_2+k_2-1) \right)\left( \Phi_1,0 \right)\\
  &+r_1^2  \left( k_1(m_1+k_1-1) - k_2(m_2+k_2-1) \right)\left( 0, \Phi_2 \right).\nonumber\\
  =&\left( \frac{1}{r_1^2}\left|\textnormal{d}\Phi_1 \right|^2 - \frac{1}{r_2^2}\left|\textnormal{d}\Phi_2 \right|^2\right)\left( (r_1^2-1)\Phi_1,r_1^2\Phi_2 \right).\nonumber
\end{align}

From Equation \eqref{ec-70} it follows that $\tau(\varphi) = 0$ if and only if 
\begin{equation}\label{ec-71}
\frac{1}{r_1^2}\left|\textnormal{d}\Phi_1 \right|^2 = \frac{1}{r_2^2}\left|\textnormal{d}\Phi_2 \right|^2.
\end{equation}

Using Equations \eqref{ec-69} and \eqref{ec-70} it follows that 
\begin{equation}\label{ec-72}
\tau_2(\varphi) = (2r_1^2-1)\tau(\varphi).
\end{equation}
Therefore, from Equations \eqref{ec-71} and \eqref{ec-72} it follows that $\varphi$ is proper biharmonic if and only if $r_1 =r_2 = 1/\sqrt{2}$ and $\left|\textnormal{d}\Phi_1 \right|^2 \neq \left|\textnormal{d}\Phi_2 \right|^2$, that is, in this case, equivalent to $e\left(\varphi_1 \right) \neq e\left( \varphi_2 \right)$.
 
\end{proof}

\begin{remark}
Although our setting involves harmonic \(k\)-forms, the resulting statements about sphere radii and energy density bear resemblance to Theorem~2.3 in \cite{O03} and Theorem~3.11 in \cite{CMO02}, where immersions and submersions are studied, respectively. Our methods are independent of those references, but the conclusions reflect a comparable geometric pattern.
\end{remark}

\end{document}